\documentclass[final]{siamart190516}
\usepackage{url}
\usepackage{amssymb}
\usepackage{amsmath}
\usepackage{stmaryrd}
\usepackage{siunitx}
\usepackage{commath}
\usepackage{graphicx}
\usepackage{subcaption}
\usepackage{enumerate}
\usepackage{hyperref}
\usepackage[numbers,sort]{natbib}
\usepackage{cleveref}

\newsiamremark{remark}{Remark}
\newsiamremark{example}{Example}


\usepackage{xspace,color}

\title{On fixed-point, Krylov, and $2\times 2$ block\\preconditioners for nonsymmetric problems
  \thanks{\funding{BS is supported by the Department of Energy, National Nuclear Security Administration,
  under Award Number(s) DE-NA0002376. SR gratefully acknowledges support from the Natural
      Sciences and Engineering Research Council of Canada through the
      Discovery Grant program (RGPIN-05606-2015) and the Discovery
      Accelerator Supplement (RGPAS-478018-2015).}}}

\author{Ben S. Southworth\thanks{Department of Applied Mathematics,
    University of Colorado,
    U.S.A. (\url{ben.s.southworth@gmail.com}),
    \url{http://orcid.org/0000-0002-0283-4928}} \and Abdullah
  A. Sivas\thanks{Department of Applied Mathematics, University of
    Waterloo, Canada (\url{aasivas@uwaterloo.ca})} \and Sander
  Rhebergen\thanks{Department of Applied Mathematics, University of
    Waterloo, Canada (\url{srheberg@uwaterloo.ca}),
    \url{http://orcid.org/0000-0001-6036-0356}}}

\headers{On fixed-point, Krylov, and block preconditioners}
{ B. S. Southworth, A. A. Sivas, and S. Rhebergen}
\begin{document}
\allowdisplaybreaks

\maketitle
\begin{abstract}
The solution of matrices with $2\times 2$ block structure arises in
numerous areas of computational mathematics, such as PDE
discretizations based on mixed-finite element methods, constrained
optimization problems, or the implicit or steady state treatment of
any system of PDEs with multiple dependent variables. Often, these
systems are solved iteratively using Krylov methods and some form of
block preconditioner. Under the assumption that one diagonal block
is inverted exactly, this paper proves a direct equivalence between
convergence of $2\times2$ block preconditioned Krylov or fixed-point
iterations to a given tolerance, with convergence of the underlying
preconditioned Schur-complement problem. In particular, results
indicate that an effective Schur-complement preconditioner is a
necessary and sufficient condition for rapid convergence of
$2\times 2$ block-preconditioned GMRES, for arbitrary
relative-residual stopping tolerances. A number of corollaries and
related results give new insight into block preconditioning, such as
the fact that approximate block-LDU or symmetric block-triangular
preconditioners offer minimal reduction in iteration over
block-triangular preconditioners, despite the additional
computational cost. Theoretical results are verified numerically on
a nonsymmetric steady linearized Navier--Stokes discretization,
which also demonstrate that theory based on the assumption of an
exact inverse of one diagonal block extends well to the more
practical setting of inexact inverses.
\end{abstract}

\section{Introduction}

\subsection{Problem}

This paper considers block preconditioning and the corresponding
convergence of fixed-point and Krylov methods applied to nonsymmetric
systems of the form
\begin{equation}
  \label{eq:system}
  A\mathbf{x} = \mathbf{b},
  \qquad \mathbf{x},\mathbf{b} \in \mathbb{R}^{n \times n},
  \quad A \in \mathbb{R}^{n\times n},
\end{equation}
where the matrix $A$ has a $2\times 2$ block structure,
\begin{equation}
  \label{eq:mat}
  A =
  \begin{bmatrix}
    A_{11}  & A_{12}
    \\
    A_{21} & A_{22}
  \end{bmatrix}.
\end{equation}
Such systems arise in numerous areas, including mixed finite element
\cite{Benzi:11,cyr2012stabilization,Klawonn:1998hv,white2011block}, constraint
optimization problems
\cite{schoberl2007symmetric,Pearson:2011ha,dollar2010preconditioning},
and the solution of neutral particle transport \cite{19hetdsa}. More
generally, the discretization of just about any systems of PDEs with
multiple dependent variables can be expressed as a $2\times 2$ block
operator by the grouping of variables into two sets. Although
iterative methods for saddle-point problems, in which
$A_{22} = \mathbf{0}$, have seen extensive research, in this paper we
take a more general approach, making minimal assumptions on the
submatrices of $A$.

The primary contribution of this paper is to prove a direct
equivalence between the convergence of a block-preconditioned fixed-point
or Krylov iteration applied to \cref{eq:system}, with
convergence of a similar method applied directly to a preconditioned
Schur complement of $A$, {where the Schur complements of $A$ are
defined as $S_{11} := A_{11} - A_{12}A_{22}^{-1}A_{21}$ and
$S_{22} := A_{22}-A_{21}A_{11}^{-1}A_{12}$.}
In particular, results in this paper prove that a
good approximation to the Schur complement of the $2\times 2$ block
matrix \cref{eq:mat} is a necessary and sufficient condition for rapid
convergence of preconditioned GMRES applied to \cref{eq:system}, for
arbitrary relative residual stopping tolerances.

The main assumption in derivations here is that at least one of $A_{11}$
or $A_{22}$ is non-singular and that the action of its inverse can be
computed. Although in practice it is often not advantageous to
solve one diagonal block to numerical precision every iteration, it
\emph{is} typically the case that the inverse of at least one diagonal
block can be reliably computed using some form of iterative method,
such as multigrid.
The theory developed in this paper provides a guide for ensuring a
convergent and practical preconditioner for \cref{eq:system}. Once the
iteration and convergence are well understood, the time to solution
can be reduced by solving
the diagonal block(s) to some tolerance. Numerical results in
\cref{sec:results} demonstrate how ideas motivated by the theory,
where one block is inverted exactly, extend to inexact preconditioners.

\subsection{Previous work}
\label{sec:prev}

For nonsymmetric $2\times 2$ block operators, most theoretical results
in the literature are not necessarily indicative of practical
performance.  There is also a lack of distinction in the literature
between a Krylov convergence result and a fixed-point convergence
result, which we discuss in \cref{sec:2x2}.

Theoretical results on block preconditioning generally fall in to one
of two categories. First, are results based on the assumption that the
inverse action of the Schur complement is available, and/or results
that show an asymptotic equivalence between the preconditioned
$2\times 2$ operator and the preconditioned Schur complement. It is
shown in \cite{Ipsen:2001ui,Murphy:2000hja} that GMRES (or other
minimal residual methods) is guaranteed to converge in two or four
iterations for a block-triangular or block-diagonal preconditioned
system, respectively, when the diagonal blocks of the preconditioner
consist of a Schur complement and the respective complementary block
of $A$ ($A_{11}$ or $A_{22}$). However, computing the action of the
Schur complement inverse is generally very expensive. In
\cite{Bai:2005go}, it is shown that if the minimal polynomial of the
preconditioned Schur complement is degree $k$, then the minimal
polynomial of the preconditioned $2\times 2$ system is at most degree
$k+1$. Although this does not require the action and inverse of the
Schur complement, it is almost {never} the case that GMRES is iterated
until the true minimal polynomial is achieved. As a consequence, the
minimal polynomial equivalence also does not provide practical
information on convergence of the $2\times 2$ system, as demonstrated
in the following example.

\begin{example}
  \label{example:nonconvergenceGMRES}
  Define two matrices $A_1,A_2\in\mathbb{R}^{1000\times1000}$,
    \begin{equation*}
    A_1 :=
    \begin{bmatrix}
      I_{500} & \mathbf{0}
      \\
      \mathbf{0} & D_1
    \end{bmatrix},
    \hspace{5ex}
    A_2 :=
    \begin{bmatrix}
      I_{500} & \mathbf{0}
      \\
      \mathbf{0} & D_2
    \end{bmatrix},
  \end{equation*}
    where $D_1\in\mathbb{R}^{500\times500}$ is tridiagonal with stencil
  $[-1,2,-1]$ and $D_2\in\mathbb{R}^{500\times500}$ is tridiagonal
  with stencil $[-1,2.0025,-1]$. Note that the minimal polynomials of
  $A_1$ and $A_2$ in the $\ell^2$-norm have degree at most $k=501$.
    \Cref{fig:res_ex} shows results from applying GMRES with no restarts
  to $A_1$ and $A_2$, with right-hand side
  $\mathbf{b} = (1,2,...,1000)^T/1000$.
  Note that neither operator reaches exact convergence in the first
  500 iterations, indicating that the minimal polynomial in both cases
  is degree $k=501$. However, despite having the same degree minimal
  polynomial (which is less than the size of the matrix), at iteration
  250, $A_2$ has reached a residual of
  $\norm{\mathbf{r}} \approx 10^{-5}$, while $A_1$ still has residual
  $\norm{\mathbf{r}} > 1$.
  \begin{figure}[!h]
    \begin{subfigure}[t]{0.475\textwidth}
    \centering
    \includegraphics[height=1.75in]{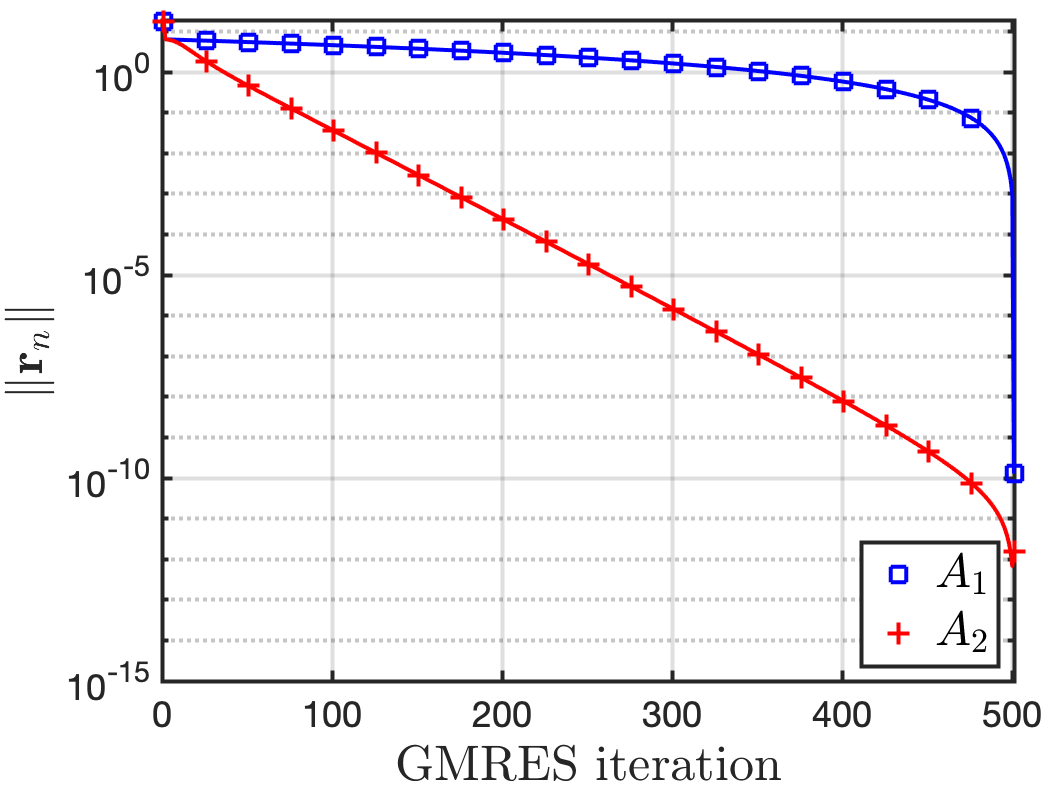}
    \caption{Minimal polynomial does not necessarily provide practical information on
    convergence of a $2\times 2$ system (\cref{example:nonconvergenceGMRES}).}
    \label{fig:res_ex}
    \end{subfigure}
    \begin{subfigure}[t]{0.475\textwidth}
    \centering
    \includegraphics[height=1.75in]{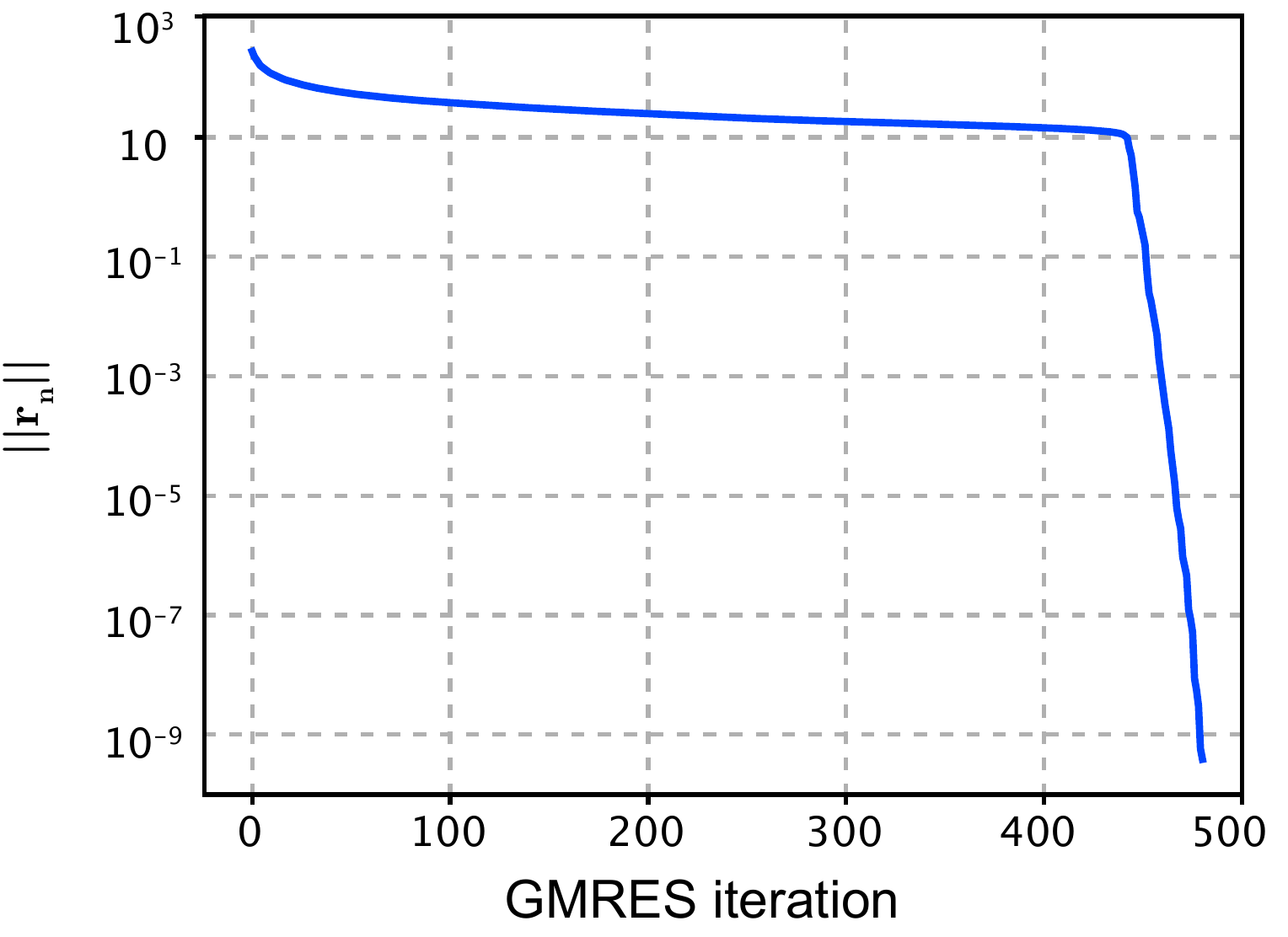}
    \caption{Eigenvalues do not necessarily provide practical information on
    convergence of a $2\times 2$ system (\cref{example:DG}).}
    \label{fig:res_dg}
    \end{subfigure}
  \end{figure}
  \end{example}

Second, many papers have used eigenvalue analyses in an attempt to 
provide more practical information on convergence. In the symmetric
setting, this has proven effective (see, for example,
\cite{Notay:2014kv}). Spectral analyses have also been done
for various nonsymmetric $2\times 2$ block matrices and
preconditioners \cite{Elman:1996kh,Klawonn:1998hv,
Bai:2006el,Bai:2005go,Krzyzanowski:2001dg,Siefert:2006fv} and
eigenvectors for preconditioned operators derived in \cite{Pestana:2014jj}.
However, eigenvalue analyses are asymptotic, guaranteeing eventual
convergence but, in the nonsymmetric setting, giving no guarantee
of practical performance. In certain cases, a nonsymmetric operator is symmetric 
in a non-standard inner product, and some of papers have looked at 
block preconditioning in modified norms 
\cite{Pestana:2012ji,Wathen:2007kp,Rees:2010kp,McDonald:2017wb}
that yield self-adjointness. 
Nevertheless, there are many nonsymmetric problems that are not easily
symmetrized and/or where eigenvalues provide little to no practical
information on convergence of iterative methods. The following
provides one such example in the discretization of differential
operators. A formal analysis as in
\cite{Greenbaum:1996cp,Tebbens:2013kv} proves that for any set
of eigenvalues, there is a matrix such that GMRES converges arbitrarily
slowly.

\begin{example}
  \label{example:DG}
  Consider an upwind discontinuous Galerkin (DG) discretization of
  linear advection with Dirichlet inflow boundaries \cite{Brezzi:2004hf}
  {and a velocity field $\mathbf{b}(x,y) := (\cos(\pi y)^2,\cos(\pi x)^2)$
  (see Figure  5a in \cite{AIR1}); more generally, similar results hold
  for any velocity field with no cycles).} In the appropriate ordering, the
  resulting matrix is block triangular, where ``block'' refers to the DG
  element blocks. Then, if we apply block-diagonal
  (Jacobi) preconditioning, the spectrum of the preconditioned
  operator is given by $\sigma(M^{-1}A) = \{1\}$, and the spectrum of
  the fixed-point iteration is given by $\sigma(I - M^{-1}A) = \{0\}$.
  Despite all zero eigenvalues, block-Jacobi preconditioned fixed-point
  or GMRES iterations on such a
  matrix can converge arbitrarily slowly, until the degree of
  nilpotency is reached and exact convergence is immediately obtained.
  {\Cref{fig:res_dg} shows convergence of DG block-Jacobi preconditioned
  GMRES applied to 2d linear transport, with $200\times 200$ finite
  elements. Convergence occurs very rapidly at around 450 iterations
  (without restart), approximately equal to the diameter of the mesh
  (as expected \cite{AIR1}).}
\end{example}

This is not the first work to recognize that eigenvalue analyses of
nonsymmetric block preconditioners may be of limited practical use.
Norm and field-of-values equivalence are known to provide more accurate
measures of convergence for nonsymmetric operators, used as early as
\cite{Manteuffel:1977ue}, and applied recently for specific problems
in \cite{Benzi:2011ku,Klawonn:1999is,Loghin:2004cp,ma2016robust}. 
Here, we  stay even more general, focusing directly on the relation
between polynomials of a general preconditioned $2\times 2$ system
and the preconditioned Schur complement.

\subsection{Overview of results}
\label{sec:intro:overview}

The paper proceeds as follows. \Cref{sec:2x2} formally introduces various
block preconditioners, considers the distinction between fixed-point and
Krylov methods, and derives some relationships on polynomials of the
preconditioned operators that define Krylov and fixed-point iterations.
Proofs and formal statements of results are
provided in \cref{sec:min}, and numerical results are examined in
\cref{sec:results}, with a discussion on the practical implications of
theory developed here. We conclude in \cref{sec:conclusions}.

Because the formal derivations are lengthy, the following list
provides a brief overview of theoretical contributions of this paper.
\begin{itemize}
\item Fixed-point and minimal residual Krylov iterations
  preconditioned with a $2\times 2$ block-triangular, block-Jacobi,
  or approximate block-LDU preconditioner converge to a given
  tolerance $C\rho$ after $n$ iterations if and only if an equivalent
  method applied to the underlying preconditioned Schur complement
  converges to tolerance $\rho$ after $n$ iterations, for constant $C$
  (\cref{sec:min}). {Such results do \textit{not} hold
  for general block-diagonal preconditioners} \cite{diag}.

\item A symmetric block-triangular or approximate block-LDU
  preconditioner offers little to no improvement in convergence over
  block-triangular preconditioners, when one diagonal block is
  inverted exactly (\cref{sec:2x2:fp:symm}). Numerical results
  demonstrate the same behaviour for inexact inverses, suggesting that
  symmetric block-triangular or block-LDU preconditioners are probably
  not worth the added computational cost in practice.

\item The worst-case number of iterations for a block-Jacobi
  preconditioner to converge to a given tolerance $\rho$ is twice the
  number of iterations for a block-triangular preconditioner to
  converge to $C\rho$, for some constant factor $C$
  (\cref{sec:min:diag}). Numerical results suggest that for non-saddle
  point problems (nonzero (2,2)-block), this double in iteration count is not due to the
  staircasing effect introduced in \cite{Fischer:1998vj} for
  saddle-point problems.

\item With an exact Schur-complement inverse, a \textit{fixed-point}
  iteration with a block-triangular preconditioner converges in two
  iterations, while a fixed-point iteration with a block-diagonal
  preconditioner does not converge (\cref{sec:2x2}).

\end{itemize}

\section{Block preconditioners}
\label{sec:2x2}

This paper considers $2\times 2$ block preconditioners, where one
diagonal block is inverted exactly, and the other is some
approximation to the Schur complement.

We consider four different kinds of block preconditioners: block
diagonal, block upper triangular, block lower triangular, and block
LDU, denoted $D$, $U$, $L$, and $M$, respectively. If the
preconditioners have no subscript, this implies the diagonal blocks of
the preconditioners are the diagonal blocks of $A$. If one of
the diagonal blocks is some approximation to the Schur complement
$S_{kk}$, $k\in\{1, 2\}$, then a $11$- or $22$- subscript denotes
in which block the approximation is used. For example, with a
Schur-complement approximation in the $(1,1)$-block, preconditioners
take the forms
\begin{align}
  \label{eq:L11U11D11M11}
  L_{11} &:=
           \begin{bmatrix}
             \widehat{S}_{11} & \mathbf{0}
             \\
             A_{21} & A_{22}
           \end{bmatrix}, 
& U_{11} &:=
           \begin{bmatrix}
             \widehat{S}_{11} & A_{12}
             \\
             \mathbf{0} & A_{22}
           \end{bmatrix}, \\ \nonumber
  D_{11} &:=
           \begin{bmatrix}
             \widehat{S}_{11} & \mathbf{0}
             \\
             \mathbf{0} & A_{22}
           \end{bmatrix},
& M_{11} &:=
           \begin{bmatrix}
             I & A_{12}A_{22}^{-1}
             \\
             \mathbf{0} & I
           \end{bmatrix}
           \begin{bmatrix}
             \widehat{S}_{11} & \mathbf{0}
             \\
             \mathbf{0} & A_{22}
           \end{bmatrix}
           \begin{bmatrix}
             I & \mathbf{0}
             \\
             A_{22}^{-1}A_{21} & I
           \end{bmatrix}.
\end{align}
The block-diagonal, block upper-triangular, and block lower-triangular
preconditioners with a Schur-complement approximation in the
$(2,2)$-block take an analogous form, with
$\widehat{S}_{11}\mapsto A_{11}$ and $A_{22}
\mapsto\widehat{S}_{22}$, and the approximate block LDU
preconditioner $M_{22}$ is given by
\begin{equation}
  \label{eq:M22}
  M_{22} := \begin{bmatrix} I & \mathbf{0} \\ A_{21}A_{11}^{-1} & I \end{bmatrix}
  	\begin{bmatrix} A_{11} & \mathbf{0} \\ \mathbf{0} & \widehat{S}_{22}\end{bmatrix}
	\begin{bmatrix} I & A_{11}^{-1}A_{12} \\ \mathbf{0} & I \end{bmatrix}.
\end{equation}
{Most results here regarding block-diagonal preconditiong are for the specific
case of block Jacobi, where $D_{11}=D_{22} = D$ is the block diagonal of $A$.}

Preconditioners are typically used in conjunction with either a
fixed-point iteration or Krylov subspace method to approximately
solve a linear system \cref{eq:system}. 
Krylov methods approximate the solution to linear systems
by constructing a Krylov space of vectors and
minimizing the error of the approximate solution over this space, in a
given norm. The Krylov space is formed as powers of the preconditioned
operator applied to the initial residual. For linear system
$A\mathbf{x} = \mathbf{b}$, (left) preconditioner $M^{-1}$, and initial
residual $\mathbf{r}_0$, the $d$th Krylov space takes the form
\begin{equation*}
  \mathcal{K}_d := \cbr{\mathbf{r}_0, M^{-1}A\mathbf{r}_0, ..., (M^{-1}A)^{d-1}\mathbf{r}_0 }.
\end{equation*}
Minimizing over this space is thus equivalent to constructing a
minimizing polynomial $p(M^{-1}A)\mathbf{r}_0$, which is optimal in a
given norm. This optimality can be in the operator norm (that is,
including a $\sup_{\mathbf{r}_0\neq\mathbf{0}}$) for a worst-case
convergence over all initial guesses and right-hand sides, or optimal
for a specific initial residual. Examples include CG, which
minimizes error in the $A$-norm, MINRES, which minimizes error in the
$AM^{-1}A$-norm \cite{Ashby:1990kr}, left-preconditioned GMRES,
which minimizes error in
the $(M^{-1}A)^*(M^{-1}A)$-norm, or right-preconditioned GMRES, which
minimizes error in the $A^*A$-norm. Note that error in the
$A^*A$-norm is equivalent to residual in the $\ell^2$-norm, which is
how minimal-residual methods are typically presented. Fixed-point
iterations also correspond to polynomials of the preconditioned
operator, but they are not necessarily optimal in a specific norm. 

Analysis in this paper is focused on polynomials of
block-preconditioned operators, particularly deriving upper and lower
bounds on minimizing Krylov polynomials of a fixed degree. \Cref{sec:2x2:fp}
begins by considering fixed-point iterations and the corresponding matrix
polynomials in the nonsymmetric setting, and discusses
important differences between the various preconditioners in
\cref{eq:L11U11D11M11} and \cref{eq:M22}. \Cref{sec:2x2:poly}
then examines general polynomials of the preconditioned operator,
developing the theoretical framework used in \cref{sec:min} to analyze
convergence of block-preconditioned Krylov methods. Due to the
equivalence of a Krylov method and a minimizing polynomial
of the preconditioned operator, we refer to, for example, GMRES and a
minimizing polynomial of $p(M^{-1}A)$ in the $\ell^2$-norm, interchangeably.

\subsection{Observations on fixed-point iterations}
\label{sec:2x2:fp}

For some approximate inverse $P$ to linear operator $A$, error
propagation of a fixed-point iteration takes the form
$\mathcal{E} := I - P^{-1}A$ and residual propagation takes the form
$\mathcal{R} := A\mathcal{E}^{-1}A = I - AP^{-1}$.  Define
\begin{align}
  \label{eq:schurFP}
  \begin{split}
    \mathcal{E}_{{11}} & := I - \widehat{S}_{11}^{-1}S_{11}, \hspace{5ex}
    \mathcal{R}_{{11}} := I -  S_{11} \widehat{S}_{11}^{-1}, \\
    \mathcal{E}_{{22}} & := I - \widehat{S}_{22}^{-1}S_{22}, \hspace{5ex}
    \mathcal{R}_{{22}} := I -  S_{22} \widehat{S}_{22}^{-1}.
  \end{split}
\end{align}
Consider first block-triangular and approximate block-LDU preconditioners. 
Powers of fixed-point error and residual propagation with these block
preconditioners take the following forms:
\begin{align}
  \label{eq:errorrespropops}
  \begin{split}   
&(I - L_{11}^{-1}A)^d = \begin{bmatrix} I \\ -A_{22}^{-1}A_{21} \end{bmatrix} \mathcal{E}_{{11}}^{d-1}
	\begin{bmatrix} I - \widehat{S}_{11}^{-1}A_{11} & -\widehat{S}_{11}^{-1}A_{12} \end{bmatrix},
\\
&(I - AL_{22}^{-1})^d = \begin{bmatrix} -A_{12}\widehat{S}_{22}^{-1} \\ I - A_{22}\widehat{S}_{22}^{-1}\end{bmatrix} \mathcal{R}_{{22}}^{d-1}
	\begin{bmatrix} -A_{21}A_{11}^{-1} & I \end{bmatrix},
\\
&(I - AU_{11}^{-1})^d = \begin{bmatrix} I - A_{11}\widehat{S}_{11}^{-1} \\ -A_{21}\widehat{S}_{11}^{-1} \end{bmatrix} 
	\mathcal{R}_{{11}}^{d-1}\begin{bmatrix} I  & -A_{12}{A}_{22}^{-1}\end{bmatrix},
\\	
&(I - U_{22}^{-1}A)^d = \begin{bmatrix} -A_{11}^{-1}A_{12} \\ I \end{bmatrix} \mathcal{E}_{{22}}^{d-1}
	\begin{bmatrix} -\widehat{S}_{22}^{-1}A_{21} & I - \widehat{S}_{22}^{-1}A_{22}\end{bmatrix},
\end{split}
\\
(I - AL_{11}^{-1})^d & = \begin{bmatrix} \mathcal{R}_{{11}}^{p} & -\mathcal{R}_{{11}}^{d-1}A_{12}A_{22}^{-1}  \\
  \mathbf{0} & \mathbf{0} \end{bmatrix},
\hspace{3ex}
(I - L_{22}^{-1}A)^d = \begin{bmatrix} \mathbf{0} & -A_{11}^{-1}A_{12}\mathcal{E}_{{22}}^{d-1} \\
  \mathbf{0} & \mathcal{E}_{{22}}^d \end{bmatrix}
\nonumber\\
(I - AU_{22}^{-1})^d & = \begin{bmatrix} \mathbf{0} & \mathbf{0} \\ -\mathcal{R}_{{22}}^{d-1}A_{21}A_{11}^{-1} &
  \mathcal{R}_{{22}}^d\end{bmatrix},
\hspace{3ex}
(I - U_{11}^{-1}A)^d =  \begin{bmatrix} \mathcal{E}_{{11}}^d &  \mathbf{0} \\
  -A_{22}^{-1}A_{21}  \mathcal{E}_{{11}}^{d-1}& \mathbf{0} \end{bmatrix}, 
\nonumber\\
(I - M_{11}^{-1}A)^d & = \begin{bmatrix} \mathcal{E}_{{11}}^d & \mathbf{0} \\
	-A_{22}^{-1}A_{21}\mathcal{E}_{{11}}^d & \mathbf{0}\end{bmatrix},
\hspace{7ex}
(I - AM_{11}^{-1})^d = \begin{bmatrix} \mathcal{R}_{{11}}^d & -\mathcal{R}_{{11}}^dA_{12}A_{22}^{-1} \\
	\mathbf{0} & \mathbf{0}\end{bmatrix},
\nonumber\\
(I - M_{22}^{-1}A)^d & = \begin{bmatrix} \mathbf{0} & -A_{11}^{-1}A_{12}\mathcal{E}_{{22}}^d \\
	\mathbf{0} & \mathcal{E}_{{22}}^d \end{bmatrix},
\hspace{7ex}
      (I - AM_{22}^{-1})^d = \begin{bmatrix} \mathbf{0} & \mathbf{0} \\ -\mathcal{R}_{{22}}^d A_{21}A_{11}^{-1} & \mathcal{R}_{{22}}^d  \end{bmatrix}.\nonumber
\end{align}

Let $\norm{\cdot}$ be a given norm on $A$ and $\norm{\cdot}_c$ be
a given norm on the Schur-complement problem.\footnote{In the case of
$\ell^p$-norms, $\norm{\cdot} =\norm{\cdot}_c$, but in general, such as
for matrix-induced norms, they may be different.} Note that any of the
above fixed-point iterations is convergent in $\norm{\cdot}$ for all initial
error or residual, if and only if the corresponding Schur-complement
fixed-point iteration in \cref{eq:schurFP} is convergent in
$\norm{\cdot}_c$. Moreover, it is well-known that for block-triangular
preconditioners with an exact Schur complement, minimal residual Krylov
methods converge in two iterations \cite{Ipsen:2001ui,Murphy:2000hja}.
However, convergence in two iterations actually follows from fixed-point
convergence rather than Krylov iterations. 

  \begin{proposition}[Block triangular-preconditioners with Schur complement]
   If $\widehat{{S}}_{kk} = {S}_{kk}$, for
   $k\in\{1,2\}$, then fixed-point iteration with a (left or right)
   block upper or block lower-triangular preconditioner converges in
   two iterations.
 \end{proposition}
 \begin{proof}
   {The proof follows by noting that if $\widehat{{S}}_{kk} = {S}_{kk}$, for
   $k\in\{1,2\}$, then all terms defined in \eqref{eq:schurFP} are zero.}
 \end{proof}
 
Now consider block-diagonal preconditioners. Then,
\begin{align*}
I - D_{11}^{-1}A & = \begin{bmatrix} I - \widehat{S}_{11}^{-1}A_{11} & -\widehat{S}_{11}^{-1}A_{12} \\ -{A}_{22}^{-1}A_{21} &  \mathbf{0} \end{bmatrix} ,
	\hspace{3ex}
I - AD_{11}^{-1} = \begin{bmatrix} I - A_{11}\widehat{S}_{11}^{-1} & -A_{12}{A}_{22}^{-1} \\ -A_{21}\widehat{S}_{11}^{-1} & \mathbf{0} \end{bmatrix} , \\
I - D_{22}^{-1}A & = \begin{bmatrix} \mathbf{0} & -A_{11}^{-1}A_{12} \\ -\widehat{S}_{22}^{-1}A_{21} & I - \widehat{S}_{22}^{-1}A_{22} \end{bmatrix} ,
	\hspace{3.25ex}
I - AD_{22}^{-1} = \begin{bmatrix} \mathbf{0} & -A_{12}\widehat{S}_{22}^{-1} \\ -A_{21}A_{11}^{-1} & I - A_{22}\widehat{S}_{22}^{-1} \end{bmatrix}.
\end{align*}
If we simplify to block Jacobi (that is, $\widehat{S}_{kk} := A_{kk}$
for $k\in\{1,2\}$),
both diagonal blocks are zero, and a
closed form for powers of block-diagonal preconditioners can be
obtained for an arbitrary number of fixed-point iterations,
\begin{align}
  \begin{split}
    (I - D^{-1}A)^{2d} & = \begin{bmatrix} A_{11}^{-1}A_{12}A_{22}^{-1}A_{21} & \mathbf{0} \\ \mathbf{0}  & A_{22}^{-1}A_{21}A_{11}^{-1}A_{12} \end{bmatrix}^d,\\
    (I - AD^{-1})^{2d} & = \begin{bmatrix} A_{12}A_{22}^{-1}A_{21}A_{11}^{-1} & \mathbf{0} \\ \mathbf{0}  & A_{21}A_{11}^{-1}A_{12}A_{22}^{-1} \end{bmatrix}^d.
  \end{split}
  \label{eq:jac_p}
\end{align}
Noting that if $\widehat{S}_{11} := A_{11}$ and
$\widehat{S}_{22} := A_{22}$ in \cref{eq:schurFP}, then
\begin{align*}
(I - D^{-1}A)^{2d} = \begin{bmatrix} \mathcal{E}_{{11}} & \mathbf{0} \\ \mathbf{0}  & \mathcal{E}_{{22}} \end{bmatrix}^d,\qquad
(I - AD^{-1})^{2d} = \begin{bmatrix} \mathcal{R}_{{11}} & \mathbf{0} \\ \mathbf{0}  & \mathcal{R}_{{22}} \end{bmatrix}^d.
\end{align*}
It follows that block Jacobi converges if and only if block upper- and
lower-triangular preconditioners, with diagonal blocks given by
$A_{11}$ and $A_{22}$, both converge. Furthermore, the expected number
of iterations of block Jacobi to converge to a given tolerance are
approximately double that of the equivalent block-triangular
preconditioning, give or take some independent constant factor (e.g.,
$A_{11}^{-1}A_{12}$) from the fixed-point operators.  A
similar result is later shown for preconditioning Krylov methods with
block Jacobi (see \Cref{th:jacobi}).  This relation of twice as many
iterations for Jacobi/block-diagonal preconditioning has been noted or
observed a number of times, perhaps originally in
\cite{Fischer:1998vj} where MINRES/GMRES are proven to stall every
other iteration on saddle-point problems.

\begin{remark}[Non-convergent block-diagonal fixed-point]
  \label{rem:diag}
  As mentioned above, fixed-point iteration converges in two
  iterations for a block-triangular preconditioner if the Schur
  complement is inverted exactly. However, the same does not hold for
  block Jacobi. Let $D_{22}$ be a block diagonal preconditioner with
  $\widehat{S}_{22} := S_{22}$. In the case of a saddle-point matrix,
  say $B$, where $B_{22} = \mathbf{0}$,
    \begin{align*}
    (I - D_{22}^{-1}B)^{3d} & = \begin{bmatrix} \mathbf{0} & -B_{11}^{-1}B_{12} \\
      -S_{22}^{-1}B_{21} & I \end{bmatrix}^{3d}
                                     = (-1)^d \begin{bmatrix} -B_{11}^{-1}B_{12}S_{22}^{-1}B_{21} &
                                       \mathbf{0} \\ \mathbf{0} & I \end{bmatrix},
  \end{align*}
    where $S_{22} := -B_{21}B_{11}^{-1}B_{12}$. Here we see the
  interesting property that as we continue to iterate,
  error-propagation of block-diagonal preconditioning does not
  converge or diverge. In fact, $(I - D_{22}^{-1}B)^{3d}$ is actually
  a periodic point of period two under the matrix-valued mapping
  $(I - D_{22}^{-1}B)^3$, for $d\geq 1$. The general $2\times 2$ case
  is more complicated and does not appear to have such a
  property. However, expanding to up to four powers gave no indication
  that it would result in a convergent fixed-point iteration, as it
  does with GMRES acceleration \cite{Ipsen:2001ui}.
\end{remark}

{\color{black}
\begin{remark}[Non-optimal block-diagonal Krylov]
For $2\times 2$ systems with nonzero diagonal blocks, block-diagonal preconditioning of minimal-residual methods with an exact Schur complement does \textit{not} necessarily
converge in a fixed number of iterations, in contrast to block-triangular preconditioners
or block-diagonal preconditioners for matrices with a zero (2,2) block (for example, see
\cite{Axelsoon2006,Silvester1994} for the symmetric case, and \cite{diag} for a complete
eigenvalue decomposition of general block-diagonal preconditioned operators, with
preconditioning based on either the diagonal blocks or an exact Schur complement). 
\end{remark}
}

\subsubsection{Symmetric block-triangular preconditioners}
\label{sec:2x2:fp:symm}

The benefit of Jacobi or block-diagonal preconditioning for SPD matrices
is that they are also SPD, which permits the use of
three-term recursion relations like conjugate gradient (CG) and MINRES,
whereas block upper- or lower-triangular preconditioners are not applicable.
Approximate block-LDU preconditioners
offer one symmetric option. Another option that might be considered,
particularly by those that work in iterative or multigrid methods, is
a symmetric triangular iteration, consisting of a block-upper triangular
iteration followed by a block-lower triangular iteration (or vice versa),
akin to a symmetric (block) Gauss--Seidel sweep. Interestingly, this
does not appear to be an effective choice.
Consider a symmetric block-triangular preconditioner with
approximate Schur complement in the (2,2)-block. The preconditioner
can take two forms, depending on whether the lower or upper iteration
is done first. For example,
\begin{align*}
(I - L_{22}^{-1}A)(I - U_{22}^{-1}A) & = I - L_{22}^{-1}(L_{22} + U_{22} + A)U_{22}^{-1} A, \\
(I - U_{22}^{-1}A)(I - L_{22}^{-1}A) & = I - U_{22}^{-1}(L_{22} + U_{22} + A)L_{22}^{-1} A.
\end{align*}
Define
$\mathcal{H}^{-1}_{22} := L_{22}^{-1}(L_{22} + U_{22} + A)U_{22}^{-1}$
and
$\mathcal{G}^{-1}_{22} := U_{22}^{-1}(L_{22} + U_{22} +
A)L_{22}^{-1}$, corresponding to upper-lower and lower-upper,
symmetric preconditioners respectively.  Expanding in block form, we
see that preconditioners associated with a symmetric block-triangular
iteration are given by
\begin{align*}
\mathcal{H}^{-1}_{22} & := \begin{bmatrix} I & \mathbf{0} \\ -\widehat{S}_{22}^{-1}A_{21} & I \end{bmatrix} 
	\begin{bmatrix} A_{11}^{-1} & \mathbf{0} \\ \mathbf{0} & 2\widehat{S}_{22}^{-1} - \widehat{S}_{22}^{-1}A_{22}\widehat{S}_{22}^{-1} \end{bmatrix}
	\begin{bmatrix} I & -A_{12}\widehat{S}_{22}^{-1} \\ \mathbf{0} & I \end{bmatrix}, \\
\mathcal{G}^{-1}_{22} &:= \begin{bmatrix} I & -A_{11}^{-1}A_{12} \\ \mathbf{0} & I \end{bmatrix} \begin{bmatrix} A_{11}^{-1} & \mathbf{0} \\
	\mathbf{0} & 2\widehat{S}_{22}^{-1} - \widehat{S}_{22}^{-1}A_{22}\widehat{S}_{22}^{-1}\end{bmatrix}
	\begin{bmatrix} I & \mathbf{0} \\ -A_{21}A_{11}^{-1} & I \end{bmatrix}.
\end{align*}

Notice that each of these preconditioners can be expressed as a
certain block-LDU type preconditioner, however, it is not clear that
either would be as good as or a better preconditioner than block
LDU. In the simplest (and also fairly common) case that
$\widehat{S}_{22} = A_{22}$, then $\mathcal{H}_{22}^{-1}$ and
$\mathcal{G}_{22}^{-1}$ are exactly equivalent to the two variants of
block-LDU preconditioning in \cref{eq:L11U11D11M11} and \cref{eq:M22},
respectively, with diagonal blocks used to approximate the Schur
complement. As we will see in \cref{sss:equivBTBLDU}, this is
also formally equivalent to a block-triangular preconditioner. 

Adding an approximation to the Schur complement in the (2,2)-block,
$\mathcal{G}_{22}^{-1}$, is equivalent to block-LDU preconditioning
with Schur-complement approximation in the (2,2)-block, except that
now we approximate $S_{22}^{-1}$ with the operator
$2\widehat{S}_{22}^{-1} -
\widehat{S}_{22}^{-1}A_{22}\widehat{S}_{22}^{-1}$, as opposed to
$\widehat{S}_{22}^{-1}$ in block-LDU preconditioning \cref{eq:M22}.
It is not clear if such an approach would ever be beneficial over
standard LDU, although it is possible one can construct such a
problem. For $\widehat{S}_{22}\neq A_{22}$, it is even less clear that
$\mathcal{H}_{22}^{-1}$ would make a good or better preconditioner
compared with LDU or block triangular. Analogous things can be said
about Schur-complement approximations in the (1,1)-block.  Numerical
results in \cref{sec:results} confirm these observations, where
symmetric block-triangular preconditioners offer at best a marginal
reduction in total iteration count over block upper- or lower-
triangular preconditioners, and sometimes observe worse convergence,
at the expense of several additional (approximate) inverses.

\subsection{Krylov and polynomials of the preconditioned matrix}
\label{sec:2x2:poly}

This section begins by considering polynomials applied to the approximate
block-LDU and block-triangular preconditioned operators in
\cref{sss:blockLDU} and \cref{sss:blockT}, respectively (the block-diagonal
preconditioner is discussed in \cref{sec:min:diag}). These results are
used in \cref{sss:equivBTBLDU} to construct a norm in which
fixed-point or Krylov iterations applied to approximate block-LDU or
block-triangular preconditioned operators are equivalent to the preconditioned
Schur complement. \Cref{sec:2x2:poly:bound} uses this equivalence to
motivate the key tool used in proofs provided in \cref{sec:min}.

\subsubsection{Approximate block-LDU preconditioner}
\label{sss:blockLDU}

In this section we apply a polynomial to the block-LDU preconditioned
operator. For an approximate block-LDU preconditioner with approximate
Schur complement in the $(2,2)$-block,
\begin{equation}
  \label{eq:M22A}
  M_{22}^{-1}A = \begin{bmatrix} I & -A_{11}^{-1}A_{12} \\ \mathbf{0} & I \end{bmatrix}
  \begin{bmatrix} I& \mathbf{0} \\ \mathbf{0} & \widehat{S}_{22}^{-1}S_{22}\end{bmatrix}
  \begin{bmatrix} I & A_{11}^{-1}A_{12} \\ \mathbf{0} & I \end{bmatrix}
  := P_1\begin{bmatrix} I& \mathbf{0} \\ \mathbf{0} & \widehat{S}_{22}^{-1}S_{22}\end{bmatrix}P_1^{-1}.      
\end{equation}
The three-term formula reveals the change of basis matrix, $P_1$,
between the LDU-preconditioned operator and the Schur-complement
problem. This allows us to express polynomials $p$ of the preconditioned
operator as a change of basis applied to the polynomial
of the preconditioned Schur complement and the identity,
\begin{align}
\begin{split}\label{eq:LDU22}
p(M_{22}^{-1}A) & = \begin{bmatrix} I & -A_{11}^{-1}A_{12} \\ \mathbf{0} & I \end{bmatrix}
	\begin{bmatrix} p(I) & \mathbf{0} \\ \mathbf{0} & p(\widehat{S}_{22}^{-1}S_{22})\end{bmatrix}
	\begin{bmatrix} I & A_{11}^{-1}A_{12} \\ \mathbf{0} & I \end{bmatrix} \\
& = \begin{bmatrix} p(I) & A_{11}^{-1}A_{12} \left(p(I) - p(\widehat{S}_{22}^{-1}S_{22})\right) \\ \mathbf{0} & p(\widehat{S}_{22}^{-1}S_{22})\end{bmatrix}.
\end{split}
\end{align}
Using right preconditioning, the polynomial takes the form
\begin{align*}
p(AM_{22}^{-1}) & = \begin{bmatrix} I & \mathbf{0} \\ A_{21}A_{11}^{-1} & I \end{bmatrix}
	\begin{bmatrix} p(I) & \mathbf{0} \\ \mathbf{0} & p(S_{22}\widehat{S}_{22}^{-1}) \end{bmatrix}
	\begin{bmatrix} I & \mathbf{0} \\ -A_{21}A_{11}^{-1} & I \end{bmatrix}\\
& = \begin{bmatrix} p(I) & \mathbf{0} \\ \left(p(I) - p(S_{22}\widehat{S}_{22}^{-1})\right)A_{21}A_{11}^{-1} & p(S_{22}\widehat{S}_{22}^{-1})\end{bmatrix}.
\end{align*}
Similarly, polynomials of the left and right preconditioned operator
by a block LDU with approximate Schur complement in the $(1,1)$-block
take the form
\begin{align}
\begin{split}\label{eq:LDU11}
p(M_{11}^{-1}A) & = \begin{bmatrix} p(\widehat{S}_{11}^{-1}S_{11}) & \mathbf{0} \\
	A_{22}^{-1}A_{21}\left(p(I) - p(\widehat{S}_{11}^{-1}S_{11})\right)& p(I) \end{bmatrix}, \\
p(AM_{11}^{-1}) & = \begin{bmatrix} p(S_{11}\widehat{S}_{11}^{-1}) & \left(p(I) - p(S_{11}\widehat{S}_{11}^{-1})\right)A_{12}A_{22}^{-1} \\
	\mathbf{0} & p(I) \end{bmatrix}.
\end{split}
\end{align}

\subsubsection{Block-triangular preconditioner}
\label{sss:blockT}

We now consider polynomials of a block-triangular preconditioned
operator. Notice that error- and residual-propagation operators for
four of the block-triangular preconditioners in
\cref{eq:errorrespropops} take a convenient form, with two zero blocks
in the $2\times 2$ matrix. We focus on these operators in
particular, looking at the left and right preconditioned operators
\begin{align*}
U_{11}^{-1}A & =  \begin{bmatrix} \widehat{S}_{11}^{-1}S_{11} &  \mathbf{0} \\
	A_{22}^{-1}A_{21} & I  \end{bmatrix}, \hspace{5ex}
AL_{11}^{-1} = \begin{bmatrix} S_{11}\widehat{S}_{11}^{-1} & A_{12}A_{22}^{-1}  \\
	\mathbf{0} &I \end{bmatrix}, \\
L_{22}^{-1}A & = \begin{bmatrix} I & A_{11}^{-1}A_{12} \\
	\mathbf{0} & \widehat{S}_{22}^{-1}S_{22} \end{bmatrix}, \hspace{5ex}
AU_{22}^{-1} = \begin{bmatrix} I & \mathbf{0} \\ A_{21}A_{11}^{-1} &
	S_{22}\widehat{S}_{22}^{-1}\end{bmatrix}.
\end{align*}
These block triangular operators are easy to raise to powers; for
example,
\begin{align}\label{eq:tri11}
(U_{11}^{-1}A)^d & =  \begin{bmatrix} (\widehat{S}_{11}^{-1}S_{11})^d &  \mathbf{0} \\
	A_{22}^{-1}A_{21}\sum_{\ell=0}^{d-1} (\widehat{S}_{11}^{-1}S_{11})^\ell & I  \end{bmatrix},
\end{align}
with similar block structures for $(AL_{11}^{-1})^d$,
$(L_{22}^{-1}A)^d$, and $(AU_{22}^{-1})^d$.

Now consider some polynomial $p(t)$ of degree $d$ with coefficients
$\{\alpha_i\}$ applied to the preconditioned operator. Diagonal blocks
are given by the polynomial directly applied to the diagonal blocks,
in this case $p(\widehat{S}_{11}^{-1}S_{11})$ and $p(I)$. One
off-diagonal block will be zero and the other (for $p(U_{11}^{-1}A)$)
takes the form $A_{22}^{-1}A_{21}F$, where
\begin{equation}
  \label{eq:F0}
  F := \sum_{i=1}^d \alpha_i \sum_{\ell=0}^{i-1} (\widehat{S}_{11}^{-1}S_{11})^\ell.
\end{equation}
Assume that $p(t)$ is a consistent polynomial, $p(0) = 1$, as is the
case in Krylov or fixed-point iterations.  Then $\alpha_0 = 1$, and
\begin{align}
  \label{eq:F}
\begin{split}
F(I - \widehat{S}_{11}^{-1}S_{11}) & = \sum_{i=1}^d \alpha_i I - \sum_{i=1}^d \alpha_i(\widehat{S}_{11}^{-1}S_{11})^i \\
& = p(I) - I - (p(\widehat{S}_{11}^{-1}S_{11}) - I) \\
& = p(I) - p(\widehat{S}_{11}^{-1}S_{11}).
\end{split}
\end{align}
If $I - \widehat{S}_{11}^{-1}S_{11}$ is invertible, not uncommon in
practice as preconditioning often does not invert any particular
eigenmode exactly, then
{\small
\begin{align}\label{eq:pTri}
p(U_{11}^{-1}A) & = \begin{bmatrix} I - \widehat{S}_{11}^{-1}S_{11} & \mathbf{0} \\ \mathbf{0} & I \end{bmatrix} 
	\begin{bmatrix} p(\widehat{S}_{11}^{-1}S_{11}) &  \mathbf{0} \\ A_{22}^{-1}A_{21}\left(p(I) - p(\widehat{S}_{11}^{-1}S_{11})\right) & p(I)  \end{bmatrix}
	\begin{bmatrix} (I - \widehat{S}_{11}^{-1}S_{11})^{-1} & \mathbf{0} \\ \mathbf{0} & I \end{bmatrix}.
\end{align}}
Analogous derivations hold for other block-triangular preconditioners.

\subsubsection{Equivalence of block-triangular and LDU preconditioners}
\label{sss:equivBTBLDU}

Notice from \cref{eq:LDU11} that the middle term in \cref{eq:pTri}
exactly corresponds to $p(M_{11}^{-1}A)$. Applying similar techniques
to the other triangular preconditioners above yield the following
result on equivalence between consistent polynomials of approximate block-LDU
preconditioned and block-triangular preconditioned operators. In
particular, this applies to polynomials resulting from fixed-point or
Krylov iterations.

\begin{proposition}[Similarity of LDU and triangular preconditioning]
  \label{prop:sim}
  Let $p(t)$ be some consistent polynomial. Then
    \begin{align*}
    p(U_{11}^{-1}A)\begin{bmatrix} I - \widehat{S}_{11}^{-1}S_{11} & \mathbf{0} \\ \mathbf{0} & I \end{bmatrix} 
	& = \begin{bmatrix} I - \widehat{S}_{11}^{-1}S_{11} & \mathbf{0} \\ \mathbf{0} & I \end{bmatrix}p(M_{11}^{-1}A) , \\
p(L_{22}^{-1}A) \begin{bmatrix} I & \mathbf{0} \\ \mathbf{0} & I - \widehat{S}_{22}^{-1}S_{22}\end{bmatrix} 
	& =\begin{bmatrix} I & \mathbf{0} \\ \mathbf{0} & I - \widehat{S}_{22}^{-1}S_{22} \end{bmatrix}  p(M_{22}^{-1}A), \\
\begin{bmatrix} I - S_{11}\widehat{S}_{11}^{-1} & \mathbf{0} \\ \mathbf{0} & I \end{bmatrix} p(AL_{11}^{-1}) 
	& = p(AM_{11}^{-1})\begin{bmatrix} I - S_{11}\widehat{S}_{11}^{-1} & \mathbf{0} \\ \mathbf{0} & I \end{bmatrix} , \\
\begin{bmatrix} I & \mathbf{0} \\ \mathbf{0} & I - S_{22}\widehat{S}_{22}^{-1}\end{bmatrix} p(AU_{22}^{-1}) 
	& = p(AM_{22}^{-1})\begin{bmatrix} I & \mathbf{0} \\ \mathbf{0} & I - S_{22}\widehat{S}_{22}^{-1} \end{bmatrix}.
  \end{align*}
    If the Schur-complement fixed-point, for example
  $I - \widehat{S}_{11}^{-1}S_{11}$, is invertible, then the above
  equalities are similarity relations between a consistent polynomial applied to an
  LDU-preconditioned operator and a block-triangular preconditioned operator.
\end{proposition}
\begin{proof}
  The proof follows from derivations analogous to those in
  \cref{sss:blockLDU} and \cref{sss:blockT}.
\end{proof}

Combining with a three-term representation of block LDU
preconditioners yields the change of basis matrix between block
triangular preconditioner operators and the preconditioned Schur
complement. For example, consider $p(L_{22}^{-1}A)$.  From
\cref{eq:M22A} and \cref{prop:sim},
\begin{equation*}
Qp(L_{22}^{-1}A) = p\left(\begin{bmatrix} I & \mathbf{0} \\ \mathbf{0} & \widehat{S}_{22}^{-1}S_{22} \end{bmatrix}\right) Q , 
	\hspace{3ex}\textnormal{where}\hspace{3ex}
Q := \begin{bmatrix} I & A_{11}^{-1}A_{12}(I - \widehat{S}_{22}^{-1}S_{22}) \\ \mathbf{0} & I - \widehat{S}_{22}^{-1}S_{22} \end{bmatrix}. 
\end{equation*}
If we suppose that $I - \widehat{S}_{22}^{-1}S_{22}$ is invertible,
then $Q$ is invertible and we can construct the norm in which
fixed-point or Krylov iterations applied to $L_{22}^{-1}A$ are
equivalent to the preconditioned Schur complement. For any consistent
polynomial $p(t)$,
{\small
\begin{equation*}
\|p(L_{22}^{-1}A)\| = \left\|p\left(\begin{bmatrix} I & \mathbf{0} \\ \mathbf{0} & \widehat{S}_{22}^{-1}S_{22} \end{bmatrix}\right)\right\|_{(QQ^*)^{-1}},
\hspace{2ex}
\left\|p\left(\begin{bmatrix} I & \mathbf{0} \\ \mathbf{0} & \widehat{S}_{22}^{-1}S_{22} \end{bmatrix}\right)\right\| = \|p(L_{22}^{-1}A)\|_{Q^*Q}.
\end{equation*}
}
Similar results are straightforward to derive for $p(U_{11}^{-1}A)$,
$p(AL_{11}^{-1})$, and $p(AU_{22}^{-1})$.

\subsubsection{On bounding minimizing Krylov polynomials}
\label{sec:2x2:poly:bound}

To motivate the framework used for most of the proofs to follow in
\cref{sec:min}, consider block-LDU preconditioning (for example,
\cref{eq:LDU11}). Observe that a polynomial $p(t)$ of the
preconditioned operator is a block-triangular matrix consisting of
combinations of $p(t)$ applied to the preconditioned Schur complement,
and $p(I)$. A natural way to bound a minimizing polynomial from above
is to then define
\begin{equation}
  \label{eq:q}
  q(t) := \varphi(t)(1-t),
\end{equation}
for some consistent polynomial $\varphi(t)$. Applying $q$ to the
preconditioned operator eliminates the identity terms, and we are left
with, for example, terms involving
$\varphi(\widehat{S}^{-1}S)(I - \widehat{S}^{-1}S)$. This is just one
fixed-point iteration applied to the preconditioned Schur complement,
and some other consistent polynomial applied to the preconditioned
Schur complement, which we can choose to be a certain minimizing
polynomial.

\Cref{prop:sim} shows that such an approximation is also convenient
for block triangular preconditioning.  The $(1-t)$ term applies the
appropriate transformation to the off-diagonal term as in
\cref{eq:F}. As in the case of block-LDU preconditioning, we are
then left with a block triangular matrix, with terms consisting of
$\varphi$ applied to the preconditioned Schur complement.

In terms of notation, in this paper $\varphi^{(d)}$ denotes some form
of minimizing polynomial, with superscript $(d)$ indicating the
polynomial degree $d$. Subscripts, e.g., $\varphi_{22}^{(d)}$,
indicate a minimizing polynomial for the corresponding
(preconditioned) $(2,2)$-Schur complement, and $q$ denotes a
polynomial of the form in \cref{eq:q}.

\section{Minimizing Krylov polynomials}
\label{sec:min}

This section uses the relations derived in \cref{sec:2x2:poly} to
prove a relation between the Krylov minimizing polynomial for the
preconditioned $2\times 2$ operator and that for the preconditioned
Schur complement. Approximate
block-LDU preconditioning is analyzed in \cref{sec:min:ldu}, followed
by block-triangular {preconditioning} in \cref{sec:min:tri}, and
block-Jacobi {preconditioning} in \cref{sec:min:diag}. As mentioned
previously, the Krylov method, such as left-preconditioned GMRES, is
referred to interchangeably with the equivalent minimizing polynomial.

\subsection{Approximate block-LDU preconditioning}
\label{sec:min:ldu}

This section first considers approximate block-LDU preconditioning and
GMRES in \Cref{th:ldu_gmres}, proving equivalence between minimizing
polynomials of the $2\times 2$ preconditioned operator and the
preconditioned Schur complement. Although we are primarily interested
in nonsymmetric operators in this paper (and thus not CG), it is
demonstrated in \Cref{th:ldu_cg} that analogous techniques can be
applied to analyze preconditioned CG. Due to the induced matrix norm
used in CG, the key step is in deriving a reduced Schur-complement
induced norm on the preconditioned Schur complement
problem.

\begin{theorem}[Block-LDU preconditioning and GMRES]
  \label{th:ldu_gmres}
  Let $\varphi^{(d)}$ denote a minimizing polynomial of the
  preconditioned operator of degree $d$ in the $\ell^2$-norm, for
  initial residual $\mathbf{r} = [\mathbf{r}_1; \mathbf{r}_2]$ (or
  initial preconditioned residual for right preconditioning). Let
  $\varphi_{kk}^{(d)}$ be the minimizing polynomial for
  $\widehat{S}_{kk}^{-1}S_{kk}$ in the $\ell^2$-norm, for initial
  residual $\mathbf{r}_k$, and $k\in\{1,2\}$. Then, {\small
        \begin{align*}
      \|\varphi_{11}^{(d)}(\widehat{S}_{11}^{-1}S_{11})\mathbf{r}_1\| \leq \|\varphi^{(d)}(M_{11}^{-1}A)\mathbf{r}\|
	& \leq \left\|\begin{bmatrix} I \\ -A_{22}^{-1}A_{21}\end{bmatrix} (I - \widehat{S}_{11}^{-1}S_{11})\varphi_{11}^{(d-1)}(\widehat{S}_{11}^{-1}S_{11})\mathbf{r}_1\right\|, \\
\frac{1}{\sqrt{2}}\left\|\varphi_{11}^{(d)}(S_{11}\widehat{S}_{11}^{-1})\hat{\mathbf{r}}_1\right\| \leq
	\|\varphi^{(d)}(AM_{11}^{-1})\mathbf{r}\| &
	\leq \left\|(I - S_{11}\widehat{S}_{11}^{-1})\varphi_{11}^{(d-1)}(S_{11}\widehat{S}_{11}^{-1})\hat{\mathbf{r}}_1\right\|, \\
\|\varphi_{22}^{(d)}(\widehat{S}_{22}^{-1}S_{22})\mathbf{r}_2\| \leq \|\varphi^{(d)}(M_{22}^{-1}A)\mathbf{r}\| 
	& \leq \left\|\begin{bmatrix} -A_{11}^{-1}A_{12} \\ I \end{bmatrix} (I - \widehat{S}_{22}^{-1}S_{22})\varphi_{22}^{(d-1)}
	(\widehat{S}_{22}^{-1}S_{22})\mathbf{r}_2\right\| , \\
\frac{1}{\sqrt{2}}\left\|\varphi_{22}^{(d)}(S_{22}\widehat{S}_{22}^{-1})\hat{\mathbf{r}}_2\right\| \leq
	\|\varphi^{(d)}(AM_{22}^{-1})\mathbf{r}\| &
	\leq \left\|(I - S_{22}\widehat{S}_{22}^{-1})\varphi_{22}^{(d-1)}(S_{22}\widehat{S}_{22}^{-1}) \hat{\mathbf{r}}_2\right\|.
    \end{align*}
  }
    where
  $\hat{\mathbf{r}}_1 := \mathbf{r}_1 - A_{12}A_{22}^{-1}\mathbf{r}_2$
  and
  $\hat{\mathbf{r}}_2 := \mathbf{r}_2 - A_{21}A_{11}^{-1}\mathbf{r}_1$.

  Now let $\varphi^{(d)}$ and $\varphi_{kk}^{(d)}$ denote minimizing
  polynomials of degree $d$ over all vectors in the
  $\ell^2$-norm. Then,
    {\small
    \begin{align*}
\|\varphi_{11}^{(d)}(\widehat{S}_{11}^{-1}S_{11})\| \leq \|\varphi^{(d)}(M_{11}^{-1}A)\|
	& \leq \left\|\begin{bmatrix} I \\ -A_{22}^{-1}A_{21}\end{bmatrix}\right\|
	\left\|(I - \widehat{S}_{11}^{-1}S_{11})\varphi_{11}^{(d-1)}(\widehat{S}_{11}^{-1}S_{11})\right\|, \\
\|\varphi_{11}^{(d)}(\widehat{S}_{11}^{-1}S_{11})\| \leq 
	\|\varphi^{(d)}(AM_{11}^{-1})\| & \leq \left\|\begin{bmatrix} I & -A_{12}A_{22}^{-1} \end{bmatrix} \right\|
	\left\|(I - S_{11}\widehat{S}_{11}^{-1})\varphi_{11}^{(d-1)}(S_{11}\widehat{S}_{11}^{-1})\right\| , \\
\|\varphi_{22}^{(d)}(\widehat{S}_{22}^{-1}S_{22})\| \leq \|\varphi^{(d)}(M_{22}^{-1}A)\| 
	& \leq \left\|\begin{bmatrix} -A_{11}^{-1}A_{12} \\ I \end{bmatrix}\right\| 
	\left\|(I - \widehat{S}_{22}^{-1}S_{22})\varphi_{22}^{(d-1)} (\widehat{S}_{22}^{-1}S_{22})\right\| , \\
\|\varphi_{22}^{(d)}(\widehat{S}_{22}^{-1}S_{22})\| \leq \|\varphi^{(d)}(AM_{22}^{-1})\| 
	& \leq \left\|\begin{bmatrix} -A_{21}A_{11}^{-1} & I \end{bmatrix} \right
	\|\left\|(I - S_{22}\widehat{S}_{22}^{-1})\varphi_{22}^{(d-1)}(S_{22}\widehat{S}_{22}^{-1})\right\|.
    \end{align*}}
\end{theorem}
\begin{proof}  
  First, recall that left-preconditioned GMRES is equivalent to
  minimizing the initial residual based on a consistent polynomial in
  $M_{22}^{-1}A$.  Let $\varphi_{22}^{(d)}(t)$ be the minimizing
  polynomial of degree $d$ for $\widehat{S}_{22}^{-1}S_{22}$, where
  $\varphi(0) = 1$.  Define the degree $d+1$ polynomial
  $q(t) := \varphi_{22}^{(d)}(t)(1-t)$. Notice that $q(0) = 1$,
  $q(1) = 0$, and from \cref{eq:LDU22} we have
    \begin{equation*}
    q(M_{22}^{-1}A) = \begin{bmatrix} \mathbf{0} & - A_{11}^{-1}A_{12}q(\widehat{S}_{22}^{-1}S_{22}) \\ \mathbf{0} & q(\widehat{S}_{22}^{-1}S_{22})\end{bmatrix}.
  \end{equation*}
    Let $\varphi^{(d+1)}$ be the minimizing polynomial in $M_{22}^{-1}A$
  of degree $d+1$ for initial residual $\mathbf{r}$. Then,
    \begin{align*}
    \|\varphi^{(d+1)}(M_{22}^{-1}A)\mathbf{r}\| & \leq \|q(M_{22}^{-1}A)\mathbf{r}\| 
                                                  = \left\|\begin{bmatrix} -A_{11}^{-1}A_{12} \\ I \end{bmatrix} (I - \widehat{S}_{22}^{-1}S_{22})\varphi_{22}^{(d)}(\widehat{S}_{22}^{-1}S_{22})\mathbf{r}_2\right\|.
  \end{align*}
    Taking the {supremum} over $\mathbf{r}$ and noting that
  $\|\mathbf{r}\| \geq \|\mathbf{r}_2\|$, this immediately yields an
  ideal GMRES bound as well, where the minimizing polynomial of degree
  $d+1$ in norm, $\varphi^{(d+1)}$, is bounded via
    \begin{align*}
    \|\varphi^{(d+1)}(M_{22}^{-1}A)\| \leq \left\|\begin{bmatrix} -A_{11}^{-1}A_{12} \\ I \end{bmatrix}\right\|
    \left\|(I - \widehat{S}_{22}^{-1}S_{22})\varphi_{22}^{(d)}(\widehat{S}_{22}^{-1}S_{22})\right\|.
  \end{align*}
    
  Right-preconditioned GMRES is equivalent to the $\ell^2$-minimizing
  consistent polynomial in $AM_{22}^{-1}$ applied to the initial preconditioned
  residual. A similar proof as above for right preconditioning yields
    \begin{align*}
    \|\varphi^{(d+1)}(AM_{22}^{-1})\mathbf{r}\| & \leq \|(I - S_{22}\widehat{S}_{22}^{-1})\varphi_{22}^{(d)}
                                                  (S_{22}\widehat{S}_{22}^{-1})\hat{\mathbf{r}}_2\|, \\
    \|\varphi^{(d+1)}(AM_{22}^{-1})\| & \leq \left\| \begin{bmatrix} -A_{21}A_{11}^{-1} & I \end{bmatrix} \right\|
                                                                                          \left\|(I - S_{22}\widehat{S}_{22}^{-1})\varphi_{22}^{(d)}(S_{22}\widehat{S}_{22}^{-1})\right\|.
  \end{align*}
    where $\mathbf{r}$ now refers to the initial preconditioned residual,
  $\varphi$ refers to minimizing polynomials for $AM_{22}^{-1}$, and
  $\hat{\mathbf{r}}_2 := \mathbf{r}_2 -A_{21}A_{11}^{-1}\mathbf{r}_1$.

  For a lower bound, let $\varphi^{(d)}$ be the minimizing polynomial
  of degree $d$ in $M_{22}^{-1}A$ for $\mathbf{r}$. Then, for an
  $\ell^p$-norm with $p\in[1,\infty]$,
    \begin{align*}
    \|\varphi^{(d)}(M_{22}^{-1}A)\mathbf{r}\| &= \left\| \begin{bmatrix} \varphi^{(d)}(I)\mathbf{r}_1 + A_{11}^{-1}A_{12}\left(\varphi^{(d)}(I) -
          \varphi^{(d)}(\widehat{S}_{22}^{-1}S_{22})\right)\mathbf{r}_2 \\ \varphi^{(d)}(\widehat{S}_{22}^{-1}S_{22}) \mathbf{r}_2 \end{bmatrix} \right\| \\
                                              & \geq \|\varphi^{(d)}(\widehat{S}_{22}^{-1}S_{22})\mathbf{r}_2\| \\
                                              & \geq \|\varphi_{22}^{(d)}(\widehat{S}_{22}^{-1}S_{22})\mathbf{r}_2\|.
  \end{align*}
    This also yields an ideal GMRES bound, where the minimizing
  polynomial in norm is bounded via
  $\|\varphi^{(d)}(M_{22}^{-1}A)\| \geq
  \|\varphi_{22}^{(d)}(\widehat{S}_{22}^{-1}S_{22})\|$.  For right
  preconditioning,
    \begin{align}
    \label{eq:LDU_AM}
    \|\varphi^{(d)}(AM_{22}^{-1})\mathbf{r}\| & = \left\| \begin{bmatrix} \varphi^{(d)}(I)\mathbf{r}_1 \\
        \varphi^{(d)}(I)\mathbf{r}_1 + \varphi^{(d)}(S_{22}\widehat{S}_{22}^{-1})(\mathbf{r}_2 - A_{21}A_{11}^{-1}\mathbf{r}_1) \end{bmatrix} \right\|.
  \end{align}
    The ideal GMRES bound follows immediately by noting that the
  {supremum} over $\mathbf{r}$ is greater than or equal to setting 
  $\mathbf{r}_1 = \mathbf{0}$ and taking the {supremum} over
  $\mathbf{r}_2$, which yields
    \begin{align*}
    \|\varphi^{(d)}(AM_{22}^{-1})\| & \geq \|\varphi^{(d)}(\widehat{S}_{22}^{-1}S_{22})\| \geq \|\varphi_{22}^{(d)}(\widehat{S}_{22}^{-1}S_{22})\|.
  \end{align*}
    Then, note the identity
    \begin{align}
    \label{eq:equality}
    \begin{split}
      \left\|\begin{bmatrix} \mathbf{x} \\ \mathbf{x} + \mathbf{y}\end{bmatrix}\right\|^2 & =
      2\|\mathbf{x}\|^2 + \|\mathbf{y}\|^2 + \langle\mathbf{x},\mathbf{y}\rangle + \langle\mathbf{y},\mathbf{x}\rangle \\
      & \hspace{5ex}\geq 2\|\mathbf{x}\|^2 + \|\mathbf{y}\|^2 - 2\|\mathbf{x}\| \|\mathbf{y}\| \geq \frac{\|\mathbf{y}\|^2}{2}.
    \end{split}
  \end{align}
    Applying \cref{eq:equality} to \cref{eq:LDU_AM} with
  $\mathbf{x} := \varphi^{(d)}(I)\mathbf{r}_1$ and
  $\mathbf{y}:=
  \varphi^{(d)}(S_{22}\widehat{S}_{22}^{-1})(\mathbf{r}_2 -
  A_{21}A_{11}^{-1}\mathbf{r}_1)$ yields the lower bound on
  $\|\varphi^{(d)}(AM_{22}^{-1})\mathbf{r}\|$.

  Appealing to \cref{eq:LDU11} and analogous derivations yield similar
  results for the block-LDU preconditioner with Schur-complement
  approximation in the (1,1)-block.
                                                                                                \end{proof}

\begin{remark}[Left vs. right preconditioning]
  Interestingly, there exist vectors $\mathbf{x}$ and $\mathbf{y}$
  such that \cref{eq:equality} is tight, suggesting there may be
  specific examples where
  $\|\varphi^{(d)}(AM_{22}^{-1})\mathbf{r}\| \leq
  \left\|\varphi_{22}^{(d)}(S_{22}\widehat{S}_{22}^{-1})\hat{\mathbf{r}}_2\right\|$.
  If this is the case (rather than a flaw elsewhere in the line of
  proof), it means there are initial residuals where the preconditioned
  $2\times 2$ operator converges faster than the corresponding
  preconditioned Schur complement, a scenario that is not possible 
  with left-preconditioning.
\end{remark}

Although the focus of this paper is general nonsymmetric operators,
similar techniques as used in the proof of \Cref{th:ldu_gmres} can be
applied to analyze CG, resulting in the following theorem.

\begin{theorem}[LDU preconditioning and CG]\label{th:ldu_cg}
  Let $\varphi^{(d)}$ be a minimizing polynomial in $M_{kk}^{-1}A$, of
  degree $d$, in the $A$-norm, for initial error vector
  $\mathbf{e} = [\mathbf{e}_1; \mathbf{e}_2]$, and $k\in\{1,2\}$. Let
  $\varphi_{kk}^{(d)}$ be the minimizing polynomial for
  $\widehat{S}_{kk}^{-1}S_{kk}$ in the $S_{kk}$ norm, for initial
  error vector $\mathbf{e}_k$. Then,
    \begin{align*}
    \|\varphi_{kk}^{(d)}(\widehat{S}_{kk}^{-1}S_{kk})\mathbf{e}_k\|_{S_{kk}}
    & \leq  \|\varphi^{(d)}(M_{kk}^{-1}A)\mathbf{e}\|_A \\
    & \leq\|(I - \widehat{S}_{kk}^{-1}S_{kk})\varphi_{kk}^{(d-1)}(\widehat{S}_{kk}^{-1}S_{kk})\mathbf{e}_k\|_{S_{kk}}.
  \end{align*}
    Now, let $\varphi^{(d)}$ denote minimizing polynomials over all
  vectors in the appropriate norm ($A$-norm or $S_{kk}$-norm),
  representing worst-case CG convergence. Then,
    \begin{align*}
    \|\varphi_{kk}^{(d)}(\widehat{S}_{kk}^{-1}S_{kk})\|_{S_{kk}}
    & \leq  \|\varphi^{(d)}(M_{kk}^{-1}A)\|_A \\
    & \leq\|(I - \widehat{S}_{kk}^{-1}S_{kk})\varphi_{kk}^{(d-1)}(\widehat{S}_{kk}^{-1}S_{kk})\|_{S_{kk}}.
  \end{align*}
  \end{theorem}
\begin{proof}
  See \cref{ap:proof_ldu_cg}.
\end{proof}

From \Cref{th:ldu_cg} we note that for CG, upper and lower
inequalities prove that after $d$ iterations, the preconditioned $2\times 2$
system converges at least as accurately as $d-1$ CG iterations on the
preconditioned Schur complement, $\widehat{S}_{kk}^{-1}S_{kk}$,
plus one fixed-point iteration, and not more accurately than $d$ CG
iterations on the preconditioned Schur complement.
Because there are operators for which convergence of
fixed-point and CG are equivalent, this indicates there are cases for
which the upper and lower bounds in \Cref{th:ldu_cg} are tight. Note,
these bounds also have no dependence on the off-diagonal blocks, a
result not shared by other preconditioners and Krylov methods examined
in this paper. It is unclear if the larger upper bound in GMRES in 
\Cref{th:ldu_gmres} is a flaw in the line of proof, or if CG on
the preconditioned $2\times 2$ system can achieve slightly better
convergence (in the appropriate norm) with respect to the 
preconditioned Schur complement than GMRES.

\subsection{Block-triangular preconditioning}
\label{sec:min:tri}

In this section we consider block-triangular preconditioning. In
particular, we prove equivalence between minimizing polynomials of the
$2\times 2$ preconditioned operator and the preconditioned Schur
complement for block-triangular preconditioning. We consider
separately left preconditioning in \Cref{th:leftT_gmres} and right
preconditioning in \Cref{th:rightT_gmres}. Theorems are stated for the
preconditioners that take the simplest form in
\cref{eq:errorrespropops} (left vs. right and Schur complement in the
(1,1)- or (2,2)-block), the same as those discussed in
\cref{sec:2x2:poly}. However, note that, for example, any polynomial
$p(U_{22}^{-1}A) = U_{22}^{-1}p(AU_{22}^{-1})U_{22}$. Thus, if we
prove a result for left preconditioning with $U_{22}^{-1}$, a similar
result holds for right preconditioning, albeit with modified
constants/residual.  Such results are not stated here for the sake of
brevity.

\begin{theorem}[Left block-triangular preconditioning and GMRES]
  \label{th:leftT_gmres}  
  Let $\varphi^{(d)}$ denote a minimizing polynomial of the
  preconditioned operator of degree $d$ in the $\ell^2$-norm, for
  initial residual $\mathbf{r} = [\mathbf{r}_1; \mathbf{r}_2]$. Let
  $\varphi_{kk}^{(d)}$ be the minimizing polynomial for
  $\widehat{S}_{kk}^{-1}S_{kk}$ in the $\ell^2$-norm, for initial
  residual $\mathbf{r}_k$, and $k\in\{1,2\}$. Then,
    \begin{align*}
    \|\varphi_{11}^{(d)}(\widehat{S}_{11}^{-1}S_{11})\mathbf{r}_1\| \leq \|\varphi^{(d)}(U_{11}^{-1}A)\mathbf{r}\|
    & \leq\left\| \begin{bmatrix} I - \widehat{S}_{11}^{-1}S_{11} \\ -A_{22}^{-1}A_{21}\end{bmatrix}
    \varphi_{11}^{(d-1)}(\widehat{S}_{11}^{-1}S_{11})\mathbf{r}_1\right\|, \\
    \|\varphi_{22}^{(d)}(\widehat{S}_{22}^{-1}S_{22})\mathbf{r}_2\| \leq \|\varphi^{(d)}(L_{22}^{-1}A)\mathbf{r}\| 
    & \leq  \left\| \begin{bmatrix} -A_{11}^{-1}A_{12} \\ I - \widehat{S}_{22}^{-1}S_{22}\end{bmatrix}
    \varphi_{22}^{(d-1)}(\widehat{S}_{22}^{-1}S_{22})\mathbf{r}_2\right\|.
  \end{align*}

  Now let $\varphi^{(d)}$ and $\varphi_{kk}^{(d)}$ denote minimizing
  polynomials of degree $d$ over all vectors in the
  $\ell^2$-norm (instead of for the initial residual). Then,
    \begin{align*}
    \|\varphi_{11}^{(d)}(\widehat{S}_{11}^{-1}S_{11})\| \leq \|\varphi^{(d)}(U_{11}^{-1}A)\|
    & \leq \left\| \begin{bmatrix} I - \widehat{S}_{11}^{-1}S_{11} \\ -A_{22}^{-1}A_{21}\end{bmatrix}\right\|
    \left\|\varphi_{11}^{(d-1)}(\widehat{S}_{11}^{-1}S_{11})\right\|,  \\
    \|\varphi_{22}^{(d)}(\widehat{S}_{22}^{-1}S_{22})\| \leq \|\varphi^{(d)}(L_{22}^{-1}A)\| 
    & \leq  \left\| \begin{bmatrix} -A_{11}^{-1}A_{12} \\ I - \widehat{S}_{22}^{-1}S_{22}\end{bmatrix}\right\|
    \left\| \varphi_{22}^{(d-1)}(\widehat{S}_{22}^{-1}S_{22})\right\|.
  \end{align*}
  \end{theorem}
\begin{proof}
  Recall left-preconditioned GMRES is equivalent to the minimizing
  consistent polynomial in the $\ell^2$-norm over the preconditioned
  operator, for initial residual $\mathbf{r}$. 
  Consider lower-triangular preconditioning with an approximate Schur
  complement in the (2,2)-block,
    \begin{equation}
    L_{22}^{-1}A = \begin{bmatrix} I & A_{11}^{-1}A_{12} \\ \mathbf{0} & \widehat{S}_{22}^{-1}S_{22} \end{bmatrix}.
  \end{equation}
    Let $\varphi(t)$ be some consistent polynomial, and define a second
  consistent polynomial $q(t) := (1-t)\varphi(t)$. Plugging in
  $t = L_{22}^{-1}A$ and expanding the polynomial $\varphi$ analogous
  to the steps in \cref{eq:tri11} and \cref{eq:F0} yields
    \begin{align*}
    q(L_{22}^{-1}A) & = \begin{bmatrix} \mathbf{0} & -A_{11}^{-1}A_{12} \\ \mathbf{0} & I - \widehat{S}_{22}^{-1}S_{22} \end{bmatrix}
                                                                                        \begin{bmatrix} \varphi(I) & F \\ \mathbf{0} & \varphi(\widehat{S}_{22}^{-1}S_{22}) \end{bmatrix} \\
                    & = \begin{bmatrix} \mathbf{0} & -A_{11}^{-1}A_{12}\varphi(\widehat{S}_{22}^{-1}S_{22}) \\ 
                      \mathbf{0} & (I - \widehat{S}_{22}^{-1}S_{22})\varphi(\widehat{S}_{22}^{-1}S_{22}) \end{bmatrix},
  \end{align*}
    where $F$ is the upper left block of $q(L_{22}^{-1}A)$, similar to
  \cref{eq:F0}.
  
  Now let $\varphi^{(d)}$ be the minimizing polynomial in
  $L_{22}^{-1}A$ of degree $d$ for initial residual
  $\mathbf{r} = [\mathbf{r}_1; \mathbf{r}_2]$, and
  $\varphi_{22}^{(d)}$ be the minimizing polynomial in
  $\widehat{S}_{22}^{-1}S_{22}$ of degree $d$ for initial residual
  $\mathbf{r}_2$. Define the degree $d$ polynomial
  $q(t) := (1-t)\varphi_{22}^{(d-1)}(t)$. Then
    \begin{equation*}
    \|\varphi^{(d)}(L_{22}^{-1}A)\mathbf{r}\|
    \leq  \|q(L_{22}^{-1}A)\mathbf{r}\|
    = \left\| \begin{bmatrix} -A_{11}^{-1}A_{12} \\ I - \widehat{S}_{22}^{-1}S_{22}\end{bmatrix}
    \varphi_{22}^{(d-1)}(\widehat{S}_{22}^{-1}S_{22})\mathbf{r}_2\right\|.
  \end{equation*}
    Taking the {supremum} over $\mathbf{r}$ and appealing to the
  submultiplicative property of norms yields an upper bound on the
  minimizing polynomial in norm as well,
    \begin{align*}
    \|\varphi^{(d)}(L_{22}^{-1}A)\| & \leq \left\| \begin{bmatrix} -A_{11}^{-1}A_{12} \\ I - \widehat{S}_{22}^{-1}S_{22}\end{bmatrix}\right\|
    \left\| \varphi_{22}^{(d-1)}(\widehat{S}_{22}^{-1}S_{22})\right\|.
  \end{align*}
    
  A lower bound is also obtained in a straightforward manner for
  initial residual $\mathbf{r}$,
    \begin{align*}
    \|\varphi^{(d)}(L_{22}^{-1}A)\mathbf{r}\| 
    & =\left\|\begin{bmatrix} \varphi(I)^{(d)}\mathbf{r}_1 + F\mathbf{r}_2 \\ \varphi^{(d)}(\widehat{S}_{22}^{-1}S_{22})\mathbf{r}_2 \end{bmatrix} \right\|
    \geq \left\| \varphi^{(d)}(\widehat{S}_{22}^{-1}S_{22})\mathbf{r}_2 \right\|
    \geq \left\| \varphi_{22}^{(d)}(\widehat{S}_{22}^{-1}S_{22})\mathbf{r}_2 \right\|,
  \end{align*}
    which can immediately be extended to a lower bound on the minimizing
  polynomial in norm as well,
    \begin{align*}
    \|\varphi^{(d)}(L_{22}^{-1}A)\| & \geq \left\| \varphi_{22}^{(d)}(\widehat{S}_{22}^{-1}S_{22})\right\|.
  \end{align*}
  
  Analogous derivations yield bounds for an upper-triangular
  preconditioner with approximate Schur complement in the (1,1)-block.
\end{proof}

We next consider right block-triangular preconditioning.

\begin{theorem}[Right block-triangular preconditioning and GMRES]
  \label{th:rightT_gmres}
  Let $\varphi^{(d)}$ denote a minimizing polynomial of the
  preconditioned operator of degree $d$ in the $\ell^2$-norm, for
  initial preconditioned residual $\mathbf{r} = [\mathbf{r}_1; \mathbf{r}_2]$.
  Let $\varphi_{kk}^{(d)}$ be the minimizing polynomial for
  $\widehat{S}_{kk}^{-1}S_{kk}$ in the $\ell^2$-norm, for initial
  residual $\hat{\mathbf{r}}_k$, and $k\in\{1,2\}$. Then,
    \begin{align*}
    \|\varphi^{(d)}(AL_{11}^{-1})\mathbf{r}\|
    & \leq \left\| \varphi_{11}^{(d-1)}(S_{11}\widehat{S}_{11}^{-1})\hat{\mathbf{r}}_1\right\|, \\
    \|\varphi^{(d)}(AU_{22}^{-1})\mathbf{r}\| 
    & \leq  \left\| \varphi_{22}^{(d-1)}(S_{22}\widehat{S}_{22}^{-1})\hat{\mathbf{r}}_2\right\|,
  \end{align*}
    where
  $\hat{\mathbf{r}}_1 := (I - S_{11}\widehat{S}_{11}^{-1})\mathbf{r}_1
  - A_{12}A_{22}^{-1}\mathbf{r}_2$ and
  $\hat{\mathbf{r}}_2 := (I - S_{22}\widehat{S}_{22}^{-1})\mathbf{r}_2
  - A_{21}A_{11}^{-1}\mathbf{r}_1$.

  Now let $\varphi^{(d)}$ and $\varphi_{kk}^{(d)}$ denote minimizing
  polynomials of degree $d$ over all vectors in the $\ell^2$-norm
  (instead of for the initial preconditioned residual). Then,
    \begin{align*}
    \|\varphi_{11}^{(d)}(S_{11}\widehat{S}_{11}^{-1})\| \leq \|\varphi^{(d)}(AL_{11}^{-1})\|
    & \leq \left\| \begin{bmatrix} I - S_{11}\widehat{S}_{11}^{-1} & -A_{12}A_{22}^{-1}\end{bmatrix}\right\|
                                                                     \left\| \varphi_{11}^{(d-1)}(S_{11}\widehat{S}_{11}^{-1})\right\|,  \\
    \|\varphi_{22}^{(d)}(S_{22}\widehat{S}_{22}^{-1})\| \leq \|\varphi^{(d)}(AU_{22}^{-1})\| 
    & \leq \left\| \begin{bmatrix} -A_{21}A_{11}^{-1} & I - S_{22}\widehat{S}_{22}^{-1} \end{bmatrix}\right\|
                                                        \left\| \varphi_{22}^{(d-1)}(S_{22}\widehat{S}_{22}^{-1})\right\|.
  \end{align*}
  \end{theorem}
\begin{proof}
  Recall right-preconditioned GMRES is equivalent to the minimizing
  consistent polynomial in the $\ell^2$-norm over the right-preconditioned
  operator, for initial preconditioned residual $\mathbf{r}$. 
  Consider
    \begin{align*}
    AL_{11}^{-1} = \begin{bmatrix} S_{11}\widehat{S}_{11}^{-1} & A_{12}A_{22}^{-1} \\ \mathbf{0} & I \end{bmatrix}.
  \end{align*}
    Defining $q(t) = \varphi(t)(1-t)$ we note that
          \begin{align*}
    q(AL_{11}^{-1}) & = \varphi(AL_{11}^{-1})(I - AL_{11}^{-1}) \\
                    & = \begin{bmatrix} \varphi(S_{11}\widehat{S}_{11}^{-1}) & F \\ \mathbf{0} & \varphi(I) \end{bmatrix} 
                                                                                                 \begin{bmatrix} I - S_{11}\widehat{S}_{11}^{-1} & -A_{12}A_{22}^{-1} \\ \mathbf{0} & \mathbf{0} \end{bmatrix} \\
                    & =  \begin{bmatrix} (I - S_{11}\widehat{S}_{11}^{-1})\varphi(S_{11}\widehat{S}_{11}^{-1}) & 
                      -\varphi(S_{11}\widehat{S}_{11}^{-1})A_{12}A_{22}^{-1} \\ \mathbf{0} & \mathbf{0} \end{bmatrix}.
  \end{align*}
    Then,
    \begin{align*}
    \|\varphi^{(d)}(AL_{11}^{-1})\mathbf{r}\| & \leq \|q(AL_{11}^{-1})\mathbf{r}\| \\
                                              &= \left\| \varphi_{11}^{(d-1)}(S_{11}\widehat{S}_{11}^{-1})\left((I - S_{11}\widehat{S}_{11}^{-1})\mathbf{r}_1 -
                                                A_{12}A_{22}^{-1}\mathbf{r}_2\right)\right\|, \\
    \|\varphi^{(d)}(AL_{11}^{-1})\| & \leq \left\| \begin{bmatrix} I - S_{11}\widehat{S}_{11}^{-1} & -A_{12}A_{22}^{-1}\end{bmatrix}\right\|
                                                                                                     \left\| \varphi_{11}^{(d-1)}(S_{11}\widehat{S}_{11}^{-1})\right\|.
  \end{align*}
    
  The lower bound on the minimizing polynomial in norm is obtained by
  noting
    \begin{align*}
    \|\varphi^{(d)}(AL_{11}^{-1})\| & = \sup_{\mathbf{r}\neq\mathbf{0}} \frac{ \left\|\begin{bmatrix}
          \varphi^{(d)}(S_{11}\widehat{S}_{11}^{-1}) & F \\ \mathbf{0} & \varphi(I) \end{bmatrix}
                                                                         \begin{bmatrix} \mathbf{r}_1 \\ \mathbf{r}_2 \end{bmatrix} \right\|}{\|\mathbf{r}\|} \\
                                    &\hspace{-5ex} \geq \sup_{\mathbf{r}_1\neq\mathbf{0}} \frac{\left\|\begin{bmatrix}
                                          \varphi^{(d)}(S_{11}\widehat{S}_{11}^{-1}) & F \\ \mathbf{0} & \varphi(I) \end{bmatrix}
                                                                                                         \begin{bmatrix} \mathbf{r}_1 \\ \mathbf{0} \end{bmatrix} \right\|}{\|\mathbf{r}_1\|} 
    = \|\varphi^{(d)}_{11}(S_{11}\widehat{S}_{11}^{-1})\|.
  \end{align*}
    Analogous derivations as above yield bounds for the right upper
  triangular preconditioner with Schur complement in the (2,2)-block,
  $AU_{22}^{-1}$.

\end{proof}

As discussed {previously}, similar results as \Cref{th:rightT_gmres} hold for
preconditioning with $U_{11}^{-1}$ and $L_{22}^{-1}$. However, it is
not clear if a lower bound for specific initial residual, as proven
for block-LDU and left block-triangular preconditioning in
\Cref{th:ldu_gmres} and \Cref{th:leftT_gmres}, holds for right
block-triangular preconditioning. For block-LDU preconditioning, the
lower bound is weaker for right preconditioning, including a factor of
$1/\sqrt{2}$.

\subsection{Block-Jacobi preconditioning}
\label{sec:min:diag}

In this section we prove equivalence between minimizing
polynomials of the $2\times 2$ preconditioned operator and the
preconditioned Schur complement for block-Jacobi preconditioning.

Let $q(t)$ be some polynomial in $t$, where $q(0) = 1$. Note that $q$
can always be written equivalently as a polynomial $p(1-t)$, under the
constraint that the sum of polynomial coefficients for $p$, say
$\{\alpha_i\}$, sum to one (to enforce $q(0) = 1$). Thus let
$\varphi^{(2d)}(D^{-1}A)$ be the minimizing polynomial of degree $2d$
in the $\ell^2$-norm, and let us express this equivalently as a
polynomial $p$, where
\begin{align*}
  p(I - D^{-1}A) &:= \sum_{i=0}^{2d} \alpha_i (I - D^{-1}A)^i \\
                 & = \sum_{i=0}^{{d}} \alpha_{2i}(I - D^{-1}A)^{2i} + 
                   \sum_{i=0}^{{d}-1} \alpha_{2i+1}(I - D^{-1}A)^{2i+1} \\
                 & = \sum_{i=0}^{{d}} \alpha_{2i}(I - D^{-1}A)^{2i} +
                   (I - D^{-1}A) \sum_{i=0}^{{d}-1} \alpha_{2i+1}(I - D^{-1}A)^{2i} .
\end{align*}
From \cref{eq:jac_p}, even powers of $I - D^{-1}A$ take a block
diagonal form, and we can write
\begin{align}
  p(I - D^{-1}A) &= \begin{bmatrix} \hat{p}(A_{11}^{-1}A_{12}A_{22}^{-1}A_{21}) & \mathbf{0} \\
    \mathbf{0}  & \hat{p}(A_{22}^{-1}A_{21}A_{11}^{-1}A_{12}) \end{bmatrix} \nonumber \\ & \hspace{-8ex}
                                                                                           + \begin{bmatrix} \mathbf{0} & -A_{11}^{-1}A_{12} \\ -A_{22}^{-1}A_{21} & \mathbf{0} \end{bmatrix}
                                                                                                                                                                     \begin{bmatrix} \tilde{p}(A_{11}^{-1}A_{12}A_{22}^{-1}A_{21}) & \mathbf{0} \\
                                                                                                                                                                       \mathbf{0}  & \tilde{p}(A_{22}^{-1}A_{21}A_{11}^{-1}A_{12}) \end{bmatrix} \nonumber \\
                 & = \begin{bmatrix} \hat{p}(A_{11}^{-1}A_{12}A_{22}^{-1}A_{21}) & -A_{11}^{-1}A_{12}\tilde{p}(A_{22}^{-1}A_{21}A_{11}^{-1}A_{12}) \\
                   -A_{22}^{-1}A_{21}\tilde{p}(A_{11}^{-1}A_{12}A_{22}^{-1}A_{21})  & \hat{p}(A_{22}^{-1}A_{21}A_{11}^{-1}A_{12}) \end{bmatrix} \nonumber\\
                                                   & = \begin{bmatrix} \hat{p}(I - A_{11}^{-1}S_{11}) & -A_{11}^{-1}A_{12}\tilde{p}(I - A_{22}^{-1}S_{22}) \\
                   -A_{22}^{-1}A_{21}\tilde{p}(I - A_{11}^{-1}S_{11})  & \hat{p}(I - A_{22}^{-1}S_{22}) \end{bmatrix}, \label{eq:blockD}
                                                                                                                       \end{align}
where $\hat{p}$ and $\tilde{p}$ are degree $d$ and $d-1$ polynomials
with coefficients $\{\hat{\alpha}_i\} \mapsfrom \{\alpha_{2i}\}$ and
$\{\tilde{\alpha}_i\} \mapsfrom \{\alpha_{2i+1}\}$, respectively. This
is the primary observation leading to the proof of
\Cref{th:jacobi}. Also note the identities that for any polynomial
$q$,
\begin{subequations}
  \begin{align}
    A_{11}^{-1}A_{12}q(I - A_{22}^{-1}S_{22})
    &= q(I - A_{11}^{-1}S_{11})A_{11}^{-1}A_{12},
    \\
    A_{22}^{-1}A_{21}q(I - A_{11}^{-1}S_{11})
    &= {q}(I - A_{22}^{-1}S_{22})A_{22}^{-1}A_{21},
  \end{align}
  \label{eq:Aq_qA}
\end{subequations}
which will be used with \cref{eq:blockD} in the derivations that
follow.

\begin{theorem}[Block-Jacobi preconditioning \& ideal GMRES]
  \label{th:jacobi}
  Let $\varphi(D^{-1}A)$ be the worst-case consistent minimizing
  polynomial of degree $2d$, in the $\ell^p$-norm, $p\in[1,\infty]$,
  for $D^{-1}A$. Let $\varphi_{11}^{(d)}$ and $\varphi_{22}^{(d)}$ be
  the minimizing polynomials of degree $d$ in the same norm, for
  $A_{11}^{-1}S_{11}$ and $A_{22}^{-1}S_{22}$, respectively. Then,
    \begin{align*}
    \|\varphi(D^{-1}A)\| &\geq \frac{\min \left\{ \|\varphi_{11}^{(d)}(A_{11}^{-1}S_{11})\|, 
                           \|\varphi_{11}^{(d)}(A_{22}^{-1}S_{22})\| \right\}}{1+\min\left\{\|A_{11}^{-1}A_{12}\|, \|A_{22}^{-1}A_{21}\|\right\}}, \\
    \frac{\|\varphi(D^{-1}A)\|}{\|A_{22}^{-1}A_{21}\| + \|A_{11}^{-1}A_{12}\|} & \leq
                                                                                 \min\left\{ \|\varphi_{11}^{(d-1)}(A_{11}^{-1}S_{11})\| ,  \|\varphi_{22}^{(d-1)}(A_{22}^{-1}S_{22})\|\right\}.
  \end{align*}

  Similarly, now let $\varphi_{11}^{(d)}$ and $\varphi_{22}^{(d)}$ be the
  minimizing polynomials of degree $d$ for $S_{11}A_{11}^{-1}$ and
  $S_{22}A_{22}^{-1}$, respectively. Then,
    \begin{align*}
    \|\varphi(AD^{-1})\| &\geq \frac{\min \left\{ \|\varphi_{11}^{(d)}(S_{11}A_{11}^{-1})\|, 
                           \|\varphi_{22}^{(d)}(S_{22}A_{22}^{-1})\| \right\}}{1+\min\left\{\|A_{21}A_{11}^{-1}\|, \|A_{12}A_{22}^{-1}\|\right\}}, \\
    \frac{\|\varphi(AD^{-1})\|}{\|A_{12}A_{22}^{-1}\| + \|A_{21}A_{11}^{-1}\|} & \leq
                                                                                 \min\left\{ \|\varphi_{11}^{(d-1)}(S_{11}A_{11}^{-1})\|, \|\varphi_{22}^{(d-1)}(S_{22}A_{22}^{-1})\|\right\}.
  \end{align*}
  \end{theorem}
\begin{proof}
  Recall preconditioned GMRES is equivalent to the minimizing
  consistent polynomial in the $\ell^2$-norm over the preconditioned
  operator. We start with the lower bounds.  Let $\varphi^{(2d)}$ be
  the consistent minimizing polynomial (in norm) of degree $2d$ for
  $D^{-1}A$, and let $p(t)$ be a polynomial such that
  $p(I-D^{-1}A) = \varphi(D^{-1}A)$, where coefficients of $p$, say
  $\{\alpha_i\}$ are such that $\sum \alpha_i = 1$. From
  \cref{eq:blockD} and \cref{eq:Aq_qA},
    \begin{align*}
    \|p(I - D^{-1}A)\| & = \sup_{\mathbf{r}\neq \mathbf{0}} \frac{\left\|\begin{bmatrix} \hat{p}(I - A_{11}^{-1}S_{11}) &
          -\tilde{p}(I - A_{11}^{-1}S_{11})A_{11}^{-1}A_{12} \\ -\tilde{p}(I - A_{22}^{-1}S_{22})A_{22}^{-1}A_{21} &
          \hat{p}(I - A_{22}^{-1}S_{22}) \end{bmatrix} \begin{bmatrix}\mathbf{r}_1\\\mathbf{r}_2\end{bmatrix}\right\|}{\|\mathbf{r}\|} \\
                       &\hspace{-15ex} \geq \sup_{\mathbf{r}_1\neq \mathbf{0}} \frac{\left\|\begin{bmatrix} \hat{p}(I - A_{11}^{-1}S_{11}) &
                             -\tilde{p}(I - A_{11}^{-1}S_{11})A_{11}^{-1}A_{12} \\ -\tilde{p}(I - A_{22}^{-1}S_{22})A_{22}^{-1}A_{21} &
                             \hat{p}(I - A_{22}^{-1}S_{22}) \end{bmatrix} \begin{bmatrix}\mathbf{r}_1\\\mathbf{0}\end{bmatrix}\right\|}{\|\mathbf{r}_1\|} \\
                       &\hspace{-15ex} \geq \max\left\{ \sup_{\mathbf{r}_1\neq\mathbf{0}} \frac{\left\| \hat{p}(I - A_{11}^{-1}S_{11})\mathbf{r}_1\right\|}{\|\mathbf{r}_1\|},
                         \sup_{\mathbf{r}_1\neq\mathbf{0}} \frac{\left\| \tilde{p}(I - A_{22}^{-1}S_{22})A_{22}^{-1}A_{21}\mathbf{r}_1\right\|}{\|\mathbf{r}_1\|} \right\} \\
                       &\hspace{-15ex} \geq \max\left\{ \sup_{\mathbf{r}_1\neq\mathbf{0}} \frac{\left\| \hat{p}(I - A_{11}^{-1}S_{11})\mathbf{r}_1\right\|}{\|\mathbf{r}_1\|},
                         \sup_{A_{11}^{-1}A_{12}\tilde{\mathbf{r}}_2\neq\mathbf{0}} \frac{\left\|
                         \tilde{p}(I - A_{22}^{-1}S_{22})A_{22}^{-1}A_{21}A_{11}^{-1}A_{12}\tilde{\mathbf{r}}_2\right\|}
                         {\|A_{11}^{-1}A_{12}\tilde{\mathbf{r}}_2\|} \right\}\\
                       &\hspace{-15ex} \geq \max\left\{ \sup_{\mathbf{r}_1\neq\mathbf{0}} \frac{\left\| \hat{p}(I - A_{11}^{-1}S_{11})\mathbf{r}_1\right\|}{\|\mathbf{r}_1\|},
                         \sup_{\tilde{\mathbf{r}}_2\neq\mathbf{0}} \frac{\left\|
                         \tilde{p}(I - A_{22}^{-1}S_{22})(I - A_{22}^{-1}S_{22})\tilde{\mathbf{r}}_2\right\|}
                         {\|A_{11}^{-1}A_{12}\|\|\tilde{\mathbf{r}}_2\|} \right\}.
  \end{align*}
    Note that the step introducing the maximum in the third line holds
  for $\ell^p$-norms, $p\in[1,\infty]$.

  Now recall that to enforce $\varphi^{(2d)}(0) = 1$, it must be the
  case that coefficients of $p$ sum to one. Thus, it must be the case
  that for coefficients of $\hat{p}$ and $\tilde{p}$, say
  $\{\hat{\alpha}_i\}$ and $\{\tilde{\alpha}_i\}$,
  $\sum_i \hat{\alpha}_i + \sum_i \tilde{\alpha}_i := \hat{s} +
  \tilde{s} = 1$. Let us normalize such that each polynomial within
  the supremum has coefficients of sum one, which yields
    \begin{align}
    &\|p(I - D^{-1}A)\| \nonumber\\
    &\hspace{3ex} \geq \max\left\{ |\hat{s}| \sup_{\mathbf{r}_1\neq\mathbf{0}} \frac{\left\|
      \hat{p}(I - A_{11}^{-1}S_{11})\mathbf{r}_1\right\|}{|\hat{s}|\|\mathbf{r}_1\|},
      |\tilde{s}| \sup_{\tilde{\mathbf{r}}_2\neq\mathbf{0}} \frac{\left\|
      \tilde{p}(I - A_{22}^{-1}S_{22})(I - A_{22}^{-1}S_{22})\tilde{\mathbf{r}}_2\right\|}
      {|\tilde{s}|\|A_{11}^{-1}A_{12}\|\|\tilde{\mathbf{r}}_2\|} \right\} \nonumber\\
    &\hspace{3ex} \geq \max \left\{ |\hat{s}| \left\| \varphi_{11}^{(d)}(A_{11}^{-1}S_{11})\right\|, 
      \frac{|1 - \hat{s}|}{\|A_{11}^{-1}A_{12}\|} \left\| \varphi_{22}^{(d)}(A_{22}^{-1}S_{22})\right\| \right\}\nonumber \\
    &\hspace{3ex} := \max \left\{ |\hat{s}|C_1, \frac{|1 - \hat{s}|}{\|A_{11}^{-1}A_{12}\|}C_2\right\}, \label{eq:max}
  \end{align}
    where $\varphi_{11}^{(d)}$ is the minimizing polynomial of degree
  $d$ of $A_{11}^{-1}S_{11}$, and similarly for $\varphi_{22}^{(d)}$
  and $A_{22}^{-1}S_{22}$. In the 22-case, note that
  $\tilde{p}(I - A_{22}^{-1}S_{22})(I - A_{22}^{-1}S_{22})$ can be
  expressed as a polynomial of degree $d$ in
  $A_{22}^{-1}S_{22}$. Furthermore, in expressing the two polynomials,
  $\tilde{p}(I - A_{22}^{-1}S_{22})$ and the product
  $\tilde{p}(I - A_{22}^{-1}S_{22})(I - A_{22}^{-1}S_{22})$, as
  polynomials in $A_{22}^{-1}S_{22}$, the identity coefficients are
  equal. In particular, when scaling by $1/|\tilde{s}|$, both
  polynomials are equivalent to consistent polynomials in
  $A_{22}^{-1}S_{22}$ of degree $d-1$ and $d$, respectively. This
  allows us to bound the polynomials in $A_{22}^{-1}S_{22}$ (as well
  as $A_{11}^{-1}S_{11}$) from below using the true worst-case
  minimizing polynomial (in norm).
  
  To derive bounds for all $p$, we now minimize over $\hat{s}$. If
  $\hat{s} \geq 1$, it follows that $\|p(I - D^{-1}A)\| \geq C_1$, and
  for $\hat{s} \leq 0$, we have
  $\|p(I - D^{-1}A)\| \geq \tfrac{C_2}{\|A_{11}^{-1}A_{12}\|}$.  For
  $\hat{s}\in(0,1)$, the minimum over $\hat{s}$ of the maximum 
  in \cref{eq:max} is obtained at $\hat{s}$ such that
  $\hat{s}C_1 = \tfrac{1 - \hat{s}}{\|A_{11}^{-1}A_{12}\|}C_2$, or
  $\hat{s}_{min} :=
  \tfrac{C_2}{\|A_{11}^{-1}A_{12}\|C_1+C_2}$. Evaluating yields
    \begin{equation*}
    \|\varphi^{(2d)}(D^{-1}A)\| = \|p(I - D^{-1}A)\|
    \geq \frac{C_1C_2}{\|A_{11}^{-1}A_{12}\|C_1+C_2} 
    \geq \frac{\min\{C_1,C_2\}}{1 + \|A_{11}^{-1}A_{12}\|}.
  \end{equation*}
    An analogous proof as above, but initially setting
  $\mathbf{r}_1 = \mathbf{0}$ rather than $\mathbf{r}_2$ yields a
  similar result,
    \begin{equation*}
    \|\varphi^{(2d)}(D^{-1}A)\|
    = \|p(I - D^{-1}A)\|
    \geq \frac{\min\{C_1,C_2\}}{1 + \|A_{22}^{-1}A_{21}\|}.
  \end{equation*}
  
  Right preconditioning follows an analogous derivation, where
  $\varphi^{(2d)}(AD^{-1}) = p(I - AD^{-1})$ instead takes the form
    \begin{align*}
    p(I - AD^{-1}) & = \begin{bmatrix} \hat{p}(I - S_{11}A_{11}^{-1}) & -A_{12}A_{22}^{-1}\tilde{p}(I - S_{22}A_{22}^{-1}) \\
      -A_{21}A_{11}^{-1}\tilde{p}(I - S_{11}A_{11}^{-1})  & \hat{p}(I - S_{22}A_{22}^{-1}) \end{bmatrix}.
  \end{align*}
  
  Next, we prove the upper bounds. Similar to previously, from
  \cref{eq:blockD} we have
    \begin{align*}
    \|p(I - D^{-1}A)\| & = \left\|\begin{bmatrix} \hat{p}(I - A_{11}^{-1}S_{11}) &
	-\tilde{p}(I - A_{11}^{-1}S_{11})A_{11}^{-1}A_{12} \\ -A_{22}^{-1}A_{21}\tilde{p}(I - A_{11}^{-1}S_{11}) &
	\hat{p}(I - A_{22}^{-1}S_{22}) \end{bmatrix}\right\| \\
                       & \leq \left\|\begin{bmatrix} \hat{q}(I - A_{11}^{-1}S_{11})&-\tilde{q}(I - A_{11}^{-1}S_{11})A_{11}^{-1}A_{12} \\
                           -A_{22}^{-1}A_{21}\tilde{q}(I - A_{11}^{-1}S_{11}) & \hat{q}(I - A_{22}^{-1}S_{22}) \end{bmatrix}\right\|,
  \end{align*}
    for polynomials $\hat{q}$ and $\tilde{q}$ of degree $d$ and $d-1$
  such that coefficients satisfy
  $\sum_i \hat{\alpha}_i + \sum_i \tilde{\alpha}_i = 1$. Let
  $\hat{q} = \mathbf{0}$. Then,
    \begin{align*}
    \|p(I - D^{-1}A)\| & \leq \left\|\begin{bmatrix} \mathbf{0} &-\tilde{q}(I - A_{11}^{-1}S_{11})A_{11}^{-1}A_{12} \\
	-A_{22}^{-1}A_{21}\tilde{q}(I - A_{11}^{-1}S_{11}) & \mathbf{0}\end{bmatrix}\right\| \\
                       &\hspace{-4ex} \leq \left\|\begin{bmatrix} \mathbf{0} & \mathbf{0} \\
                           -A_{22}^{-1}A_{21}\tilde{q}(I - A_{11}^{-1}S_{11}) & \mathbf{0}\end{bmatrix}\right\| + 
                                                                                \left\|\begin{bmatrix} \mathbf{0} &-\tilde{q}(I - A_{11}^{-1}S_{11})A_{11}^{-1}A_{12} \\
                                                                                    \mathbf{0} & \mathbf{0}\end{bmatrix}\right\| \\
                       &\hspace{-4ex} = \|A_{22}^{-1}A_{21}\tilde{q}(I - A_{11}^{-1}S_{11})\| + \|\tilde{q}(I - A_{11}^{-1}S_{11})A_{11}^{-1}A_{12}\| \\
                       &\hspace{-4ex} \leq \|\tilde{q}(I - A_{11}^{-1}S_{11})\| \left(\|A_{22}^{-1}A_{21}\| + \|A_{11}^{-1}A_{12}\|\right).
  \end{align*}
    Recalling that
  $\tilde{q}(I - A_{11}^{-1}S_{11})A_{11}^{-1}A_{12} =
  A_{11}^{-1}A_{12} \tilde{q}(I - A_{22}^{-1}S_{22})$ and
  $A_{22}^{-1}A_{21}\tilde{q}(I - A_{11}^{-1}S_{11}) = \tilde{q}(I -
  A_{22}^{-1}S_{22})A_{22}^{-1}A_{21}$ for any polynomial $\tilde{q}$,
  we also have the equivalent result
    \begin{align*}
    \|p(I - D^{-1}A)\| & \leq \|\tilde{q}(I - A_{22}^{-1}S_{22})\| \left(\|A_{22}^{-1}A_{21}\| + \|A_{11}^{-1}A_{12}\|\right).
  \end{align*}
  
  Let $\varphi_{11}^{(d-1)}(t)$ and $\varphi_{22}^{(d-1)}(t)$ denote
  the consistent worst-case minimizing polynomials of degree $d-1$ for
  $A_{11}^{-1}S_{11}$ and $A_{22}^{-1}S_{22}$, respectively. Note,
  $\tilde{q}$ is also a polynomial of degree $d-1$ in (without loss of
  generality) $A_{11}^{-1}S_{11}$. Because coefficients of $\tilde{q}$
  satisfy $\sum_i \tilde{\alpha}_i = 1$,
  $\tilde{q}(I - A_{11}^{-1}S_{11})$ can equivalently be expressed as
  a consistent polynomial in $(A_{11}^{-1}S_{11})$. Thus let
  $\tilde{q}(I-A_{11}^{-1}S_{11}) :=
  \varphi_{11}^{(d-1)}(A_{11}^{-1}S_{11})$.  Analogous steps for
  $A_{22}^{-1}S_{22}$ yield bounds
    \begin{align*}
    \|\varphi^{(2d)}(D^{-1}A)\| & \leq \|\varphi_{11}^{(d-1)}(A_{11}^{-1}S_{11})\| \left(\|A_{22}^{-1}A_{21}\| + \|A_{11}^{-1}A_{12}\|\right), \\
    \|\varphi^{(2d)}(D^{-1}A)\| & \leq \|\varphi_{22}^{(d-1)}(A_{22}^{-1}S_{22})\| \left(\|A_{22}^{-1}A_{21}\| + \|A_{11}^{-1}A_{12}\|\right).
  \end{align*}
  
  Similar to the proof of a lower bound, an analogous derivation as
  above yields right preconditioning bounds
    \begin{align*}
    \|\varphi^{(2d)}(AD^{-1})\| & \leq \|\varphi_{11}^{(d-1)}(S_{11}A_{11}^{-1})\| \left(\|A_{12}A_{22}^{-1}\| + \|A_{21}A_{11}^{-1}\|\right), \\
    \|\varphi^{(2d)}(AD^{-1})\| & \leq \|\varphi_{22}^{(d-1)}(S_{22}A_{22}^{-1})\| \left(\|A_{12}A_{22}^{-1}\| + \|A_{21}A_{11}^{-1}\|\right),
  \end{align*}
    where $\varphi$ now denotes minimizing polynomials associated with
  right preconditioning.
\end{proof}

{\color{black}
\begin{remark}[General block-diagonal preconditioner]
  This section proved results for block-Jacobi
  preconditioners, where the preconditioner inverts the diagonal blocks
  of the original matrix, and convergence is defined by the underlying
  preconditioned Schur complement. However, such results do not extend
  to more general block-diagonal preconditioners with Schur-complement
  approximation $\widehat{S}_{22}\neq A_{22}$. In \cite{diag}, examples
  are constructed where block-diagonal preconditioning with an exact
  Schur complement take several hundred iterations to converge,
  while block-triangular preconditioning with an exact Schur complement
  requires only three iterations (the extra iteration over a theoretical
  max of two is likely due to floating point error).
\end{remark}
}

\section{The steady linearized Navier--Stokes equations}
\label{sec:results}

To demonstrate the new theory in practice, we consider a
finite-element discretization of the steady linearized Navier--Stokes
equations, which results in a nonsymmetric operator with block
structure, to which we apply various block-preconditioning
techniques. The finite-element discretization is constructed using the
MFEM finite-element library \cite{mfem-library}, PETSc is used for the
block-preconditioning and linear-algebra interface
\cite{petsc-user-ref}, and \textit{hypre} provides the
algebraic multigrid (AMG) solvers for various blocks in the operator
\cite{Falgout:2002vu}.

Let $\Omega \subset \mathbb{R}^2$ be a polygonal domain with boundary
$\partial\Omega$. We consider the steady linearized Navier--Stokes
problem for the velocity field $u:\Omega \to \mathbb{R}^2$ and
pressure field $p:\Omega \to \mathbb{R}$, given by
\begin{subequations}
  \begin{align}
    \label{eq:Oseen_a}
    -\nu\Delta u + \nabla\cdot(w \otimes u) - \gamma\nabla\nabla\cdot u + \nabla p
    &= f && \text{in } \Omega,
    \\
    \label{eq:Oseen_b}
    \nabla \cdot u &= 0 && \text{in } \Omega,
    \\
    \label{eq:Oseen_c}
    u &= g && \text{on } \partial\Omega,
  \end{align}
  \label{eq:Oseen}
\end{subequations}
where $w:\Omega \to \mathbb{R}^2$ is a given solenoidal velocity
field, $\nu \in \mathbb{R}^+$ is the kinematic viscosity,
$\gamma \ge 0$ is a constant, $g \in \mathbb{R}^2$ is a given Dirichlet
boundary condition, and $f : \Omega \to \mathbb{R}^d$ is a forcing
term. The consistent grad-div term $-\gamma\nabla\nabla\cdot u$ is
added {\color{black}to \cref{eq:Oseen_a}} to improve convergence of the
iterative solver when solving the discrete form of \cref{eq:Oseen}.

As a test case in this section we set {\color{black}$\gamma=1000$},
$\nu = 10^{-4}$, and $f$ and $g$ are chosen such that the exact
solution is given by
\begin{equation*}
  u =
  \begin{bmatrix}
    \sin(3 x_1)\sin(3 x_2) \\
    \cos(3 x_1)\cos(3 x_2)
  \end{bmatrix},
  \qquad
  p = (1-3x_1)x_2, \quad\quad \text{in } \Omega = [0,1]^2,
\end{equation*}
{\color{black}with $w = u$.}

We discretize the linearized Navier--Stokes problem \cref{eq:Oseen}
using the pointwise mass-conserving hybridizable discontinuous
Galerkin (HDG) method introduced in \cite{Rhebergen:2018a}. This HDG
method approximates both the velocity and pressure separately on
element interiors and element boundaries. As such, we make a
distinction between interior element degrees-of-freedom (DOFs) and the
facet DOFs. Separating the interior DOFs and facet DOFs, the HDG
linear system takes the form
\begin{equation}
  \label{eq:linearsystem}
  \begin{bmatrix}
    A_{uu}        & A_{u\bar{u}}       & B_{pu}^T & B_{\bar{p}u}^T \\
    A_{\bar{u}u}   & A_{\bar{u}\bar{u}} & 0        & 0 \\
    B_{pu}         & 0                  & 0        & 0 \\
    B_{\bar{p}u}   & 0                  & 0        & 0
  \end{bmatrix}
  \begin{bmatrix}
    u \\ \bar{u} \\ p \\ \bar{p}
  \end{bmatrix}
  =
  \begin{bmatrix}
    L \\ 0 \\ 0 \\ 0
  \end{bmatrix},
\end{equation}
where $u$ corresponds to the velocity field DOFs inside the elements,
$\bar{u}$ corresponds to the velocity DOFs on facets, and likewise for
$p$ and $\bar{p}$. The HDG method in \cite{Rhebergen:2018a} is such
that element DOFs are local; a direct consequence is that $A_{uu}$ is
a block diagonal matrix. Using this, we eliminate $u$ from
\cref{eq:linearsystem} and get the statically-condensed system
\begin{equation}
  \label{eq:reduced}
  \begin{bmatrix}
    A_{\bar{u}\bar{u}}-A_{\bar{u}u}A^{-1}_{uu}A_{u\bar{u}} & -A_{\bar{u}u}A^{-1}_{uu}B_{pu}^T & -A_{\bar{u}u}A^{-1}_{uu}B_{\bar{p}u}^T \\
    -B_{pu}A^{-1}_{uu}A_{u\bar{u}} & -B_{pu}A^{-1}_{uu}B_{pu}^T & -B_{pu}A^{-1}_{uu}B_{\bar{p}u}^T \\
    -B_{\bar{p}u}A^{-1}_{uu}A_{u\bar{u}} & -B_{\bar{p}u}A^{-1}_{uu}B_{pu}^T & -B_{\bar{p}u}A^{-1}_{uu}B_{\bar{p}u}^T
  \end{bmatrix}
  \begin{bmatrix}
    \bar{u} \\
    p \\
    \bar{p}
  \end{bmatrix}
  =
  \begin{bmatrix}
    - A_{\bar{u}u}A^{-1}_{uu}L \\
    - B_{pu}A^{-1}_{uu}L \\
    - B_{\bar{p}u}A^{-1}_{uu}L
  \end{bmatrix}.
\end{equation}
In this section we verify the theory developed in \cref{sec:2x2}
and \cref{sec:min} by solving the statically-condensed block system
\cref{eq:reduced}. Note that this is a $3\times 3$ block
system while the theory developed in this paper is for a $2 \times 2$
block system \cref{eq:system}--\cref{eq:mat}. For this reason, we lump
together the pressure DOFs $p$ and $\bar{p}$ and write
\cref{eq:reduced} in the form \cref{eq:system}--\cref{eq:mat} with
\begin{align}
  \label{eq:hdg_A11_A22}
\begin{split}
  A_{11} &= A_{\bar{u}\bar{u}}-A_{\bar{u}u}A^{-1}_{uu}A_{u\bar{u}},
  \hspace{5.75ex}
  A_{12} = 
  \begin{bmatrix}
    -A_{\bar{u}u}A^{-1}_{uu}B_{pu}^T & -A_{\bar{u}u}A^{-1}_{uu}B_{\bar{p}u}^T
  \end{bmatrix},
  \\
    A_{21} &=
  \begin{bmatrix}
    -B_{pu}A^{-1}_{uu}A_{u\bar{u}} \\ -B_{\bar{p}u}A^{-1}_{uu}A_{u\bar{u}}
  \end{bmatrix},
  \hspace{8ex}
  A_{22} =
  \begin{bmatrix}
    -B_{pu}A^{-1}_{uu}B_{pu}^T & -B_{pu}A^{-1}_{uu}B_{\bar{p}u}^T \\
    -B_{\bar{p}u}A^{-1}_{uu}B_{pu}^T & -B_{\bar{p}u}A^{-1}_{uu}B_{\bar{p}u}^T    
  \end{bmatrix},
  \end{split}
\end{align}
and
\begin{equation*}
  \mathbf{x} =
  \begin{bmatrix}
    \bar{u} \\ P
  \end{bmatrix},
  \quad
  P =
  \begin{bmatrix}
    p \\ \bar{p}
  \end{bmatrix},
  \quad
  \mathbf{b} =
  \begin{bmatrix}
    - A_{\bar{u}u}A^{-1}_{uu}L \\
    \mathbf{b}_p
  \end{bmatrix},
  \quad
  \mathbf{b}_p =
  \begin{bmatrix}
    - B_{pu}A^{-1}_{uu}L \\
    - B_{\bar{p}u}A^{-1}_{uu}L    
  \end{bmatrix}.
\end{equation*}

\subsection{Block preconditioning}
\label{sec:results:prec}

We consider block upper- and lower-triangular (respectively, $U_{22}$
and $L_{22}$), block-diagonal ($D_{22}$), approximate block-LDU
($M_{22}$), and both versions of symmetric block-triangular
preconditioners discussed in \cref{sec:2x2:fp:symm}. In all cases the
Schur complement $S_{22}$ of \cref{eq:mat} is approximated by
\begin{equation}
  \label{eq:S22_HDG}
  \widehat{S}_{22} =
  \begin{bmatrix}
    -B_{pu}A_{uu}^{-1}B_{pu}^T & \mathbf{0} \\
    \mathbf{0} & -B_{\bar{p}u}A_{uu}^{-1}B_{\bar{p}u}^T    
  \end{bmatrix}.
\end{equation}
In \cite{sivas:graddiv2020} we show that $\widehat{S}_{22}$ is a good
approximation to the corresponding Schur complement $S_{22}$. Note
that the diagonal block on $p$, $-B_{pu}A_{uu}^{-1}B_{pu}^T$, is block
diagonal and can be inverted directly. Furthermore, for
large $\gamma$ the diagonal block on $\bar{p}$,
$-B_{\bar{p}u}A_{uu}^{-1}B_{\bar{p}u}^T$, is a Poisson-like operator
which can be inverted rapidly using classical AMG techniques. Finally,
the momentum block $A_{11}$ in all preconditioners is an approximation
to an advection-diffusion equation. To this block we apply the
nonsymmetric AMG solver based on approximate ideal restriction (AIR)
\cite{AIR2,AIR1}, a recently developed nonsymmetric AMG method that is
most effective on advection-dominated problems. Altogether, we have
fast, scalable solvers for the diagonal blocks in the different
preconditioners.

Theory in this paper proves that convergence of Krylov methods applied
to the block-preconditioned system is governed by an equivalent Krylov
method applied to the preconditioned Schur complement. Since
$\widehat{S}_{22}$ is a good approximation to the corresponding Schur
complement $S_{22}$, we consider block preconditioners based on the
diagonal blocks $\{A_{11}, \widehat{S}_{22}\}$:
\begin{enumerate}
\item An (approximate) inverse of the momentum block $A_{11}$ using
  AIR and a block-diagonal inverse of the pressure block for
  $\widehat{S}_{22}$.
\item An (approximate) inverse of the momentum block $A_{11}$ using
  AIR and a negative block-diagonal inverse of the pressure block for
  $\widehat{S}_{22}$.
\end{enumerate}

The diagonal inverses computed in the pressure block
$\widehat{S}_{22}$ are solved to a small tolerance. The sign is swapped
on the pressure block, a technique often used with symmetric systems
to maintain an SPD preconditioner, to study the effect of sign of
$\widehat{S}_{22}^{-1}$ on convergence.

Although theory developed here is based on an exact inverse of the
momentum block, we present results ranging from an exact inverse to a
fairly crude inverse, with a reduction in relative residual of only
$0.1$ per iteration, and demonstrate that theoretical results extend
well to the case of inexact inverses in practice. Although AIR has
proven an effective solver for advection-dominated problems, solving
the momentum block can still be challenging. For this
reason, a relative-residual tolerance for the momentum block is used
(as opposed to doing a fixed number of iterations of AIR) as it is not
clear a priori how many iterations would be appropriate. Since this
results in a preconditioner that is different each iteration, we use
{\color{black} FGMRES acceleration (which uses right preconditioning
by definition)} \cite{saad1993flexible}. This is used as a practical
choice, and we demonstrate that the performance of
FGMRES is also consistent with theory.

\subsection{Results}\label{sec:results:results}

\Cref{fig:NS} shows iterations to various global relative-residual
tolerances as a function of relative-residual tolerance of the
momentum block for block upper- and lower-triangular, block-diagonal,
approximate block-LDU, and both versions of symmetric block triangular
preconditioners. In general, theory derived in this paper based on the
assumption of an exact inverse of one diagonal block extends
well to the inexact setting. Further points to take away from
\cref{fig:NS} are:

\begin{figure}[!ht]
  \centering
  \begin{center}
    \begin{subfigure}[t]{0.475\textwidth}
      \includegraphics[width=\textwidth]{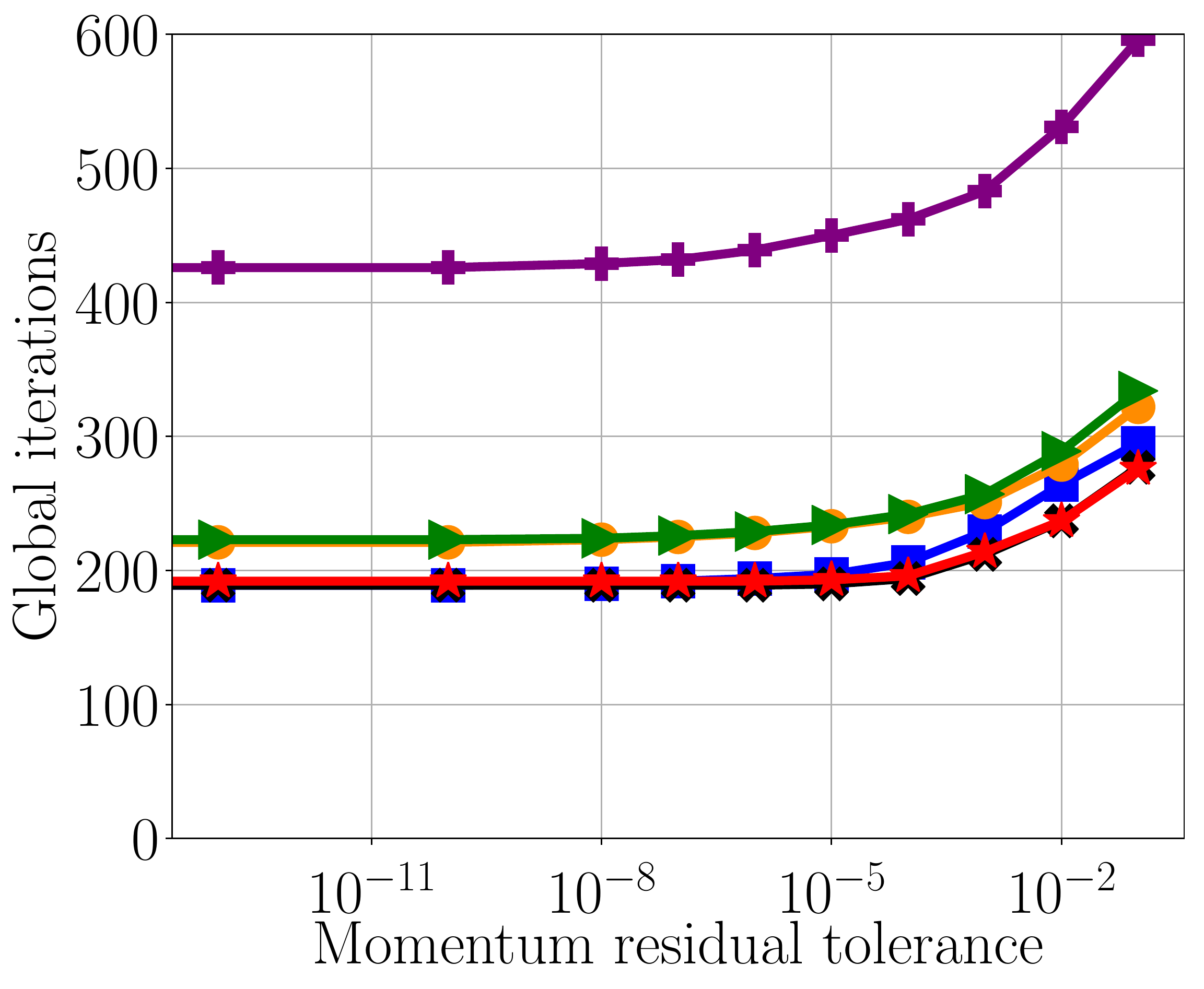}
      \caption{$10^{-11}$ relative residual tolerance.}
      \label{fig:11}
    \end{subfigure}\hspace{2ex}
    \begin{subfigure}[t]{0.47\textwidth}
      \includegraphics[width=\textwidth]{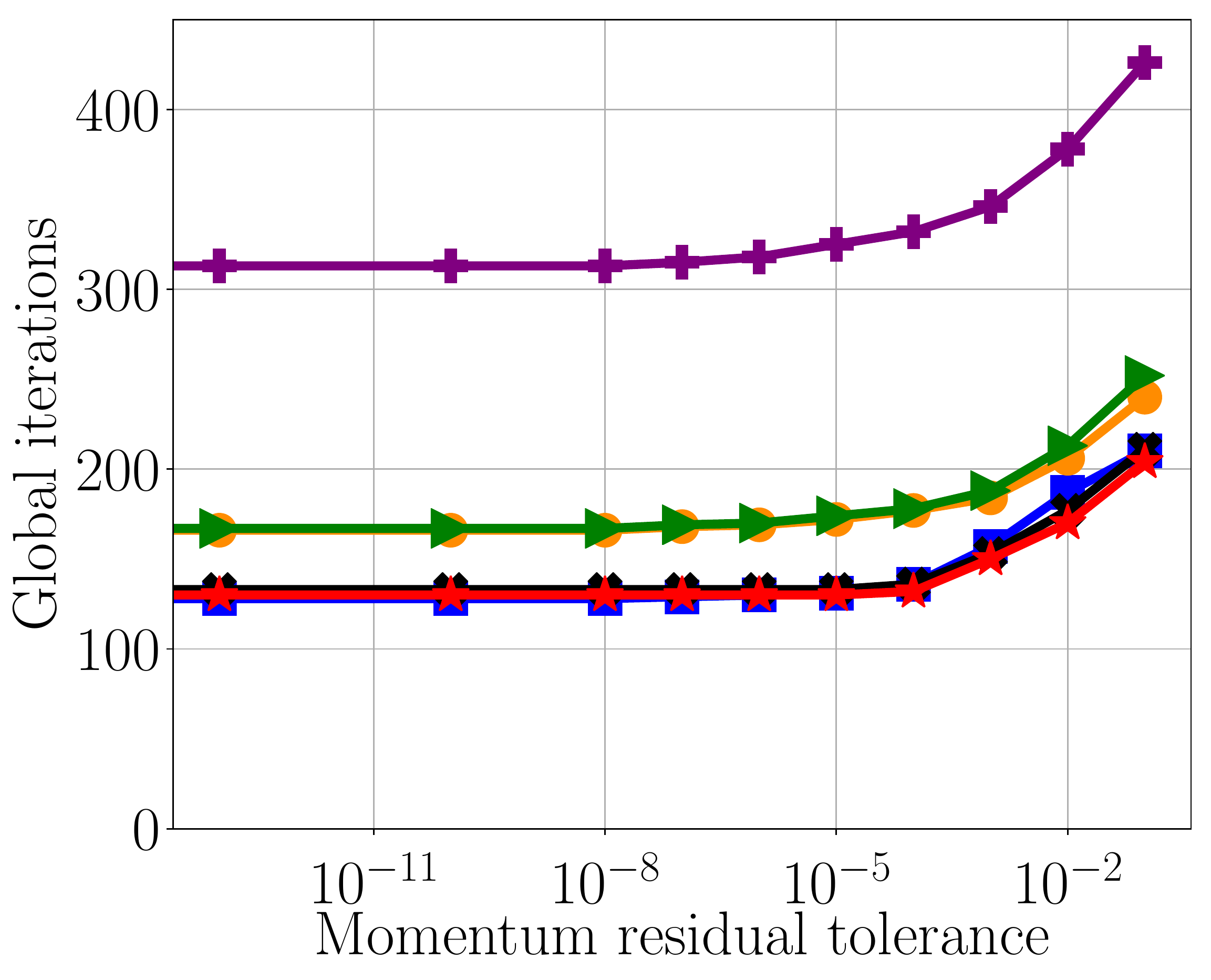}
      \caption{$10^{-8}$ relative residual tolerance.}
      \label{fig:8}
    \end{subfigure}
    \\
    \begin{subfigure}[t]{0.475\textwidth}
      \includegraphics[width=\textwidth]{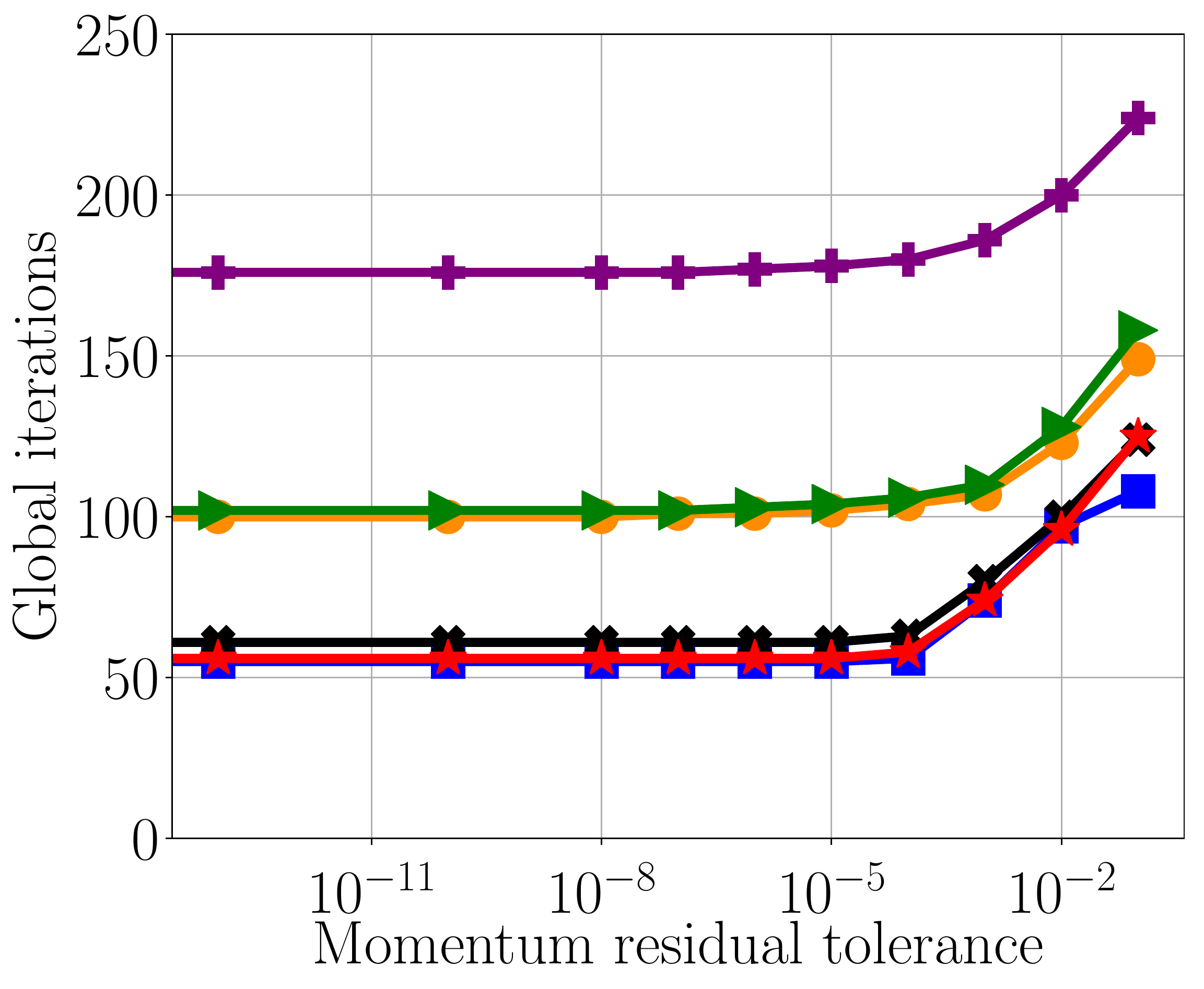}
      \caption{$10^{-5}$ relative residual tolerance.}
      \label{fig:5}
    \end{subfigure}\hspace{2ex}
    \begin{subfigure}[t]{0.47\textwidth}
      \includegraphics[width=\textwidth]{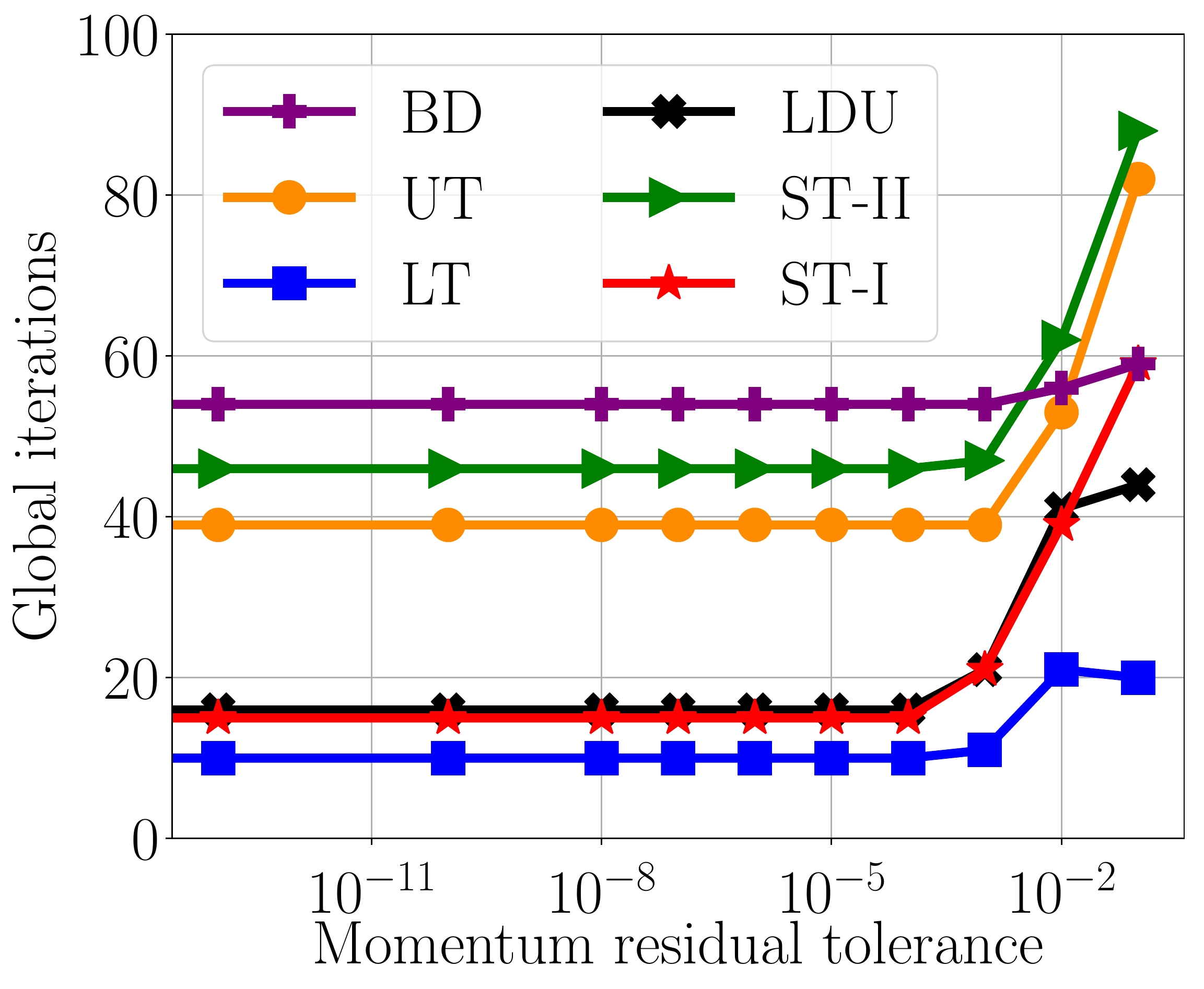}
      \caption{$10^{-3}$ relative residual tolerance.}
      \label{fig:3}
    \end{subfigure}
  \end{center}
  \caption{Number of iterations for the $2\times 2$ block
    preconditioned system to converge to $10^{-11}$, $10^{-8}$,
    $10^{-5}$, and $10^{-3}$ relative residual tolerance, as a
    function of the relative residual tolerance to solve the momentum
    block. Results are shown for block lower-triangular (LT), block
    upper-triangular (UT), symmetric lower-then-upper block-triangular
    (ST-I), symmetric upper-then-lower block-triangular (ST-II),
    block-diagonal (BD), and approximate block-LDU (LDU)
    preconditioners.}
  \label{fig:NS}
\end{figure}

\begin{enumerate}
\item For four different relative-residual tolerances of the
  $2\times 2$ block system, block-diagonal preconditioning takes very
  close to twice as many iterations as block-triangular
  preconditioning. For larger tolerances such as $10^{-3}$, it is
  approximately twice the average number of iterations of block upper-
  and block lower-triangular preconditioning, which is consistent with
  the derivations and constants in \Cref{th:jacobi}. Moreover, this relationship
  holds for almost all tolerances of the momentum block solve, with
  the exception of considering both large momentum tolerances ($>10^{-3}$)
  and large global tolerances (see \cref{fig:3}).

\item At no point does a symmetric block-triangular or approximate
  block-LDU preconditioner offer improved convergence over a
  block-triangular preconditioner, regardless of momentum or
  $2\times 2$ system residual tolerance, although the solve times are
  significantly longer due to the additional applications of the
  diagonal blocks of the preconditioner. In fact, for a global
  tolerance of $10^{-3}$ symmetric block-triangular preconditioning is
  actually less effective than just block triangular.

\item The block lower-triangular preconditioner is more effective than
  the block-upper-triangular preconditioner. However, they differ in
  iteration count by roughly the same 30--40 iterations for all four
  tolerances tested, indicating it is not a difference in convergence
  rate (which the theory says it should not be), but rather a
  difference in the leading constants. Interestingly, it cannot be
  explained by the norm of off-diagonal blocks (which are similar for
  upper- and lower-triangular preconditioning in this case). We
  hypothesize it is due to the initial residual, where for
  $\mathbf{r}^{(0)} = [\mathbf{r}_1^{(0)}, \mathbf{r}_2^{(0)}]$, we
  have $\|\mathbf{r}_1^{(0)}\| = 19.14$ and
  $\|\mathbf{r}_2^{(0)}\| = 0.0063$. Heuristically, it seems more
  effective in terms of convergence to solve directly on the block
  with the large initial residual (in this case the (1,1)-block) and
  lag the variable with a small initial residual (in this case the
  (2,2)-block), which would correspond to block-lower triangular
  preconditioning. However, a better understanding of upper vs. lower
  block-triangular preconditioning is ongoing work.
\end{enumerate}

A common approach for saddle-point problems {that are}
self-adjoint in a given inner product is to use an SPD preconditioner
so that three-term recursion formulae, in particular preconditioned
MINRES, can be used. For matrices with saddle-point structure, the
Schur complement is often negative definite, so this is achieved by
preconditioning $S_{22}$ with some approximation
$-\widehat{S}_{22}^{-1}$. Although this is advantageous in terms of
being able to use MINRES, convergence can suffer
compared with GMRES and an indefinite preconditioner.

Results of this paper indicate a
direct correlation between the minimizing polynomial for the
$2\times 2$ system and the preconditioned Schur complement. 
Moreover, convergence on the preconditioned Schur
complement should be independent of sign, because Krylov methods
minimize over a Krylov space that is invariant to the sign of
$M^{-1}$. Together, this indicates that if the (1,1)-block is inverted
exactly, convergence of GMRES applied to the $2\times2$ preconditioned
system should be approximately equivalent, regardless of sign of the
Schur-complement preconditioner.

\Cref{fig:NSm} demonstrates this property, considering {\color{black} FGMRES}
iterations on the $2\times 2$ system to relative-residual tolerances of
$10^{-11}$ and $10^{-5}$, as a function of momentum relative-residual
tolerance. Results are shown for block-diagonal, block
lower-triangular, and block-upper-triangular preconditioners, with a
natural sign $\widehat{S}_{22}^{-1}$ (solid lines) and swapped sign
$-\widehat{S}_{22}^{-1}$ (dotted lines). For accurate solves of the
momentum block, we see relatively tight convergence behaviour between
$\pm \widehat{S}_{22}^{-1}$. As the momentum solve tolerance is
relaxed, convergence of block-triangular preconditioners decay for
$-\widehat{S}_{22}^{-1}$. Interestingly, the same phenomenon does not
appear to happen for block-diagonal preconditioners, and rather there
is a fixed difference in iteration count between
$\pm \widehat{S}_{22}^{-1}$. This is likely because a block-diagonal
preconditioner does not directly couple the variables of the
$2\times 2$ matrix, while the block-triangular preconditioner does. An
inexact inverse loses a nice cancellation property of the exact
inverse, and the triangular coupling introduces terms along the lines
of $I \pm \widehat{S}_{22}^{-1}A_{22}$ {(see
  \cref{eq:errorrespropops})}, which clearly depend on the sign of
$\widehat{S}_{22}^{-1}$.

\begin{figure}[!ht]
    \centering
    \begin{center}
        \begin{subfigure}[t]{0.475\textwidth}
            \includegraphics[width=\textwidth]{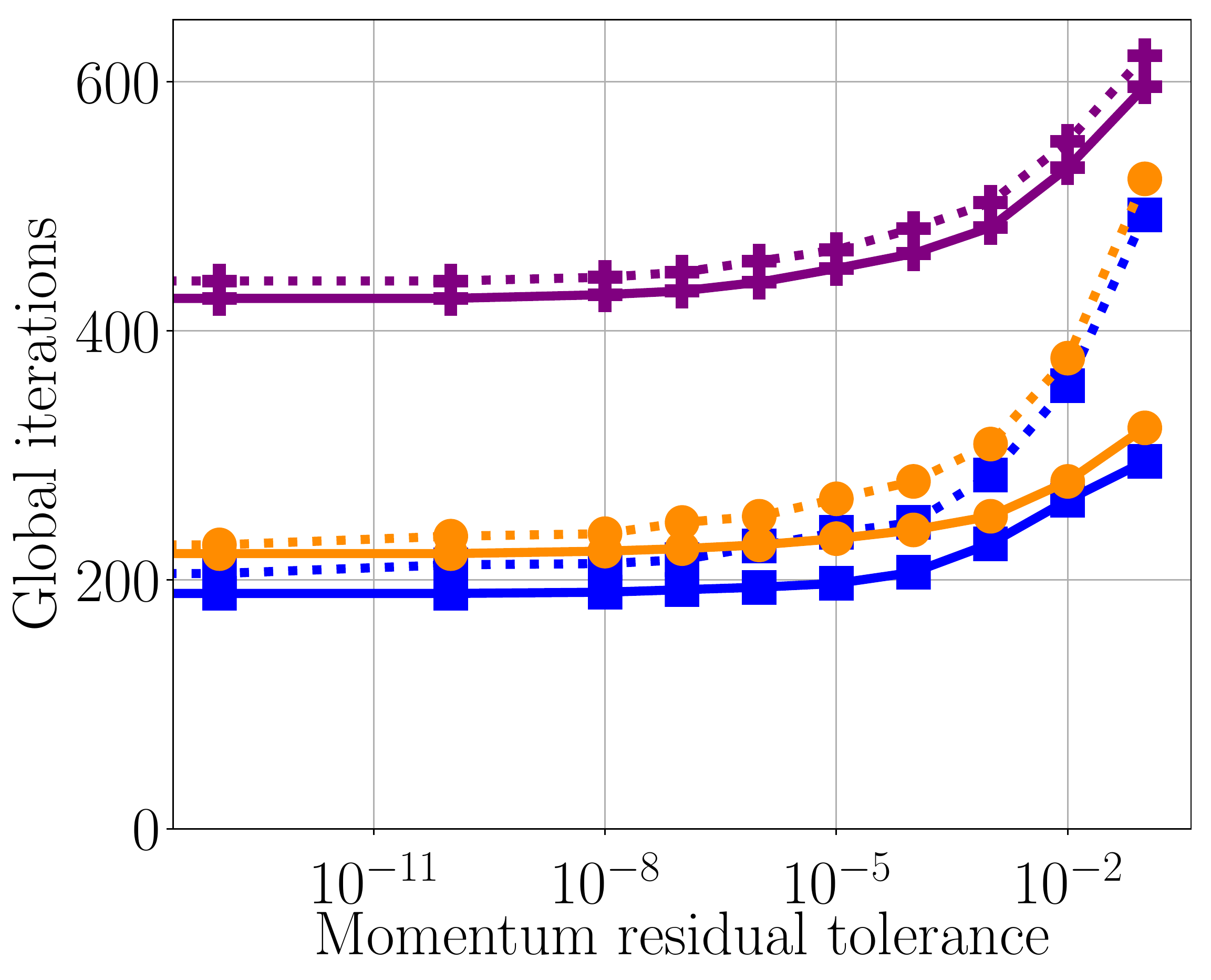}
	\caption{$10^{-11}$ relative residual tolerance.}
	\label{fig:11m}
        \end{subfigure}
        \begin{subfigure}[t]{0.475\textwidth}
            \includegraphics[width=\textwidth]{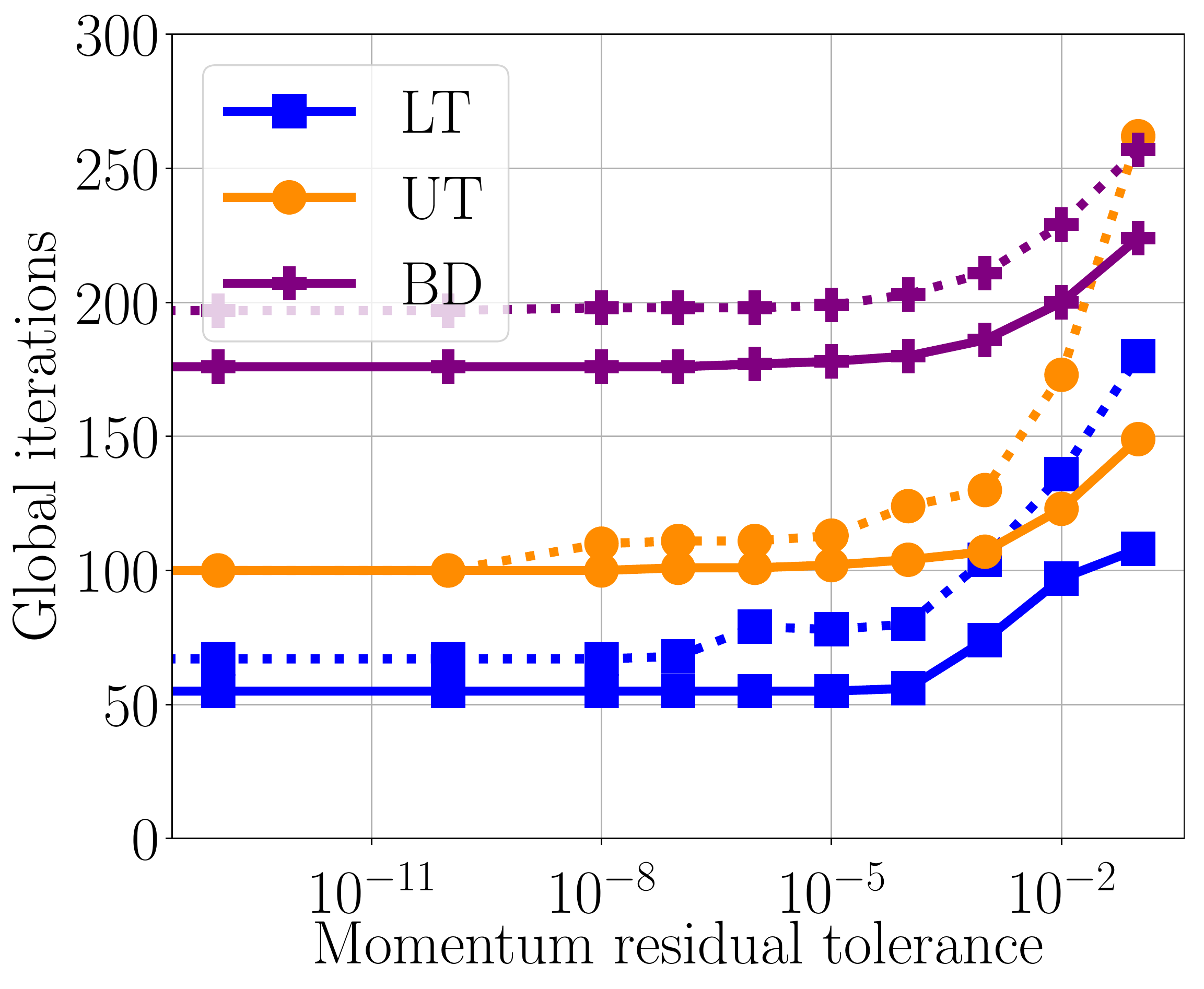}
	\caption{$10^{-5}$ relative residual tolerance.}
	\label{fig:5m}
        \end{subfigure}
    \end{center}
    \caption{Number of iterations for the $2\times 2$ block preconditioned system to converge
    to $10^{-11}$ and $10^{-5}$ relative residual tolerance, as a function
    of the relative residual tolerance to solve the momentum block. Results are shown for
    block lower-triangular (LT), block upper-triangular (UT), and block-diagonal (BD)
    preconditioners, as in \cref{fig:NS} (solid lines) and with the sign swapped 
    on the pressure Schur-complement approximation (dotted lines).}
    \label{fig:NSm}
\end{figure}

In \cite{Fischer:1998vj} it is proven that minimal residual methods
applied to saddle-point problems with a zero (2,2)-block and
preconditioned with a block-diagonal preconditioner observe a
staircasing effect, where every alternate iteration stalls. This results
in approximately twice as many iterations to convergence as a similar
block-triangular preconditioner. Although the proof appeals to
specific starting vectors, the effect is demonstrated in practice as
well. \Cref{th:jacobi} proved block-diagonal preconditioning is
expected to take twice as many iterations as block-triangular
preconditioning to reach a given tolerance (within some constant
multiplier). \Cref{fig:CF} looks at the GMRES convergence factor as a
function of iteration for block-diagonal preconditioning and block
lower-triangular preconditioning, with $\pm
\widehat{S}_{22}^{-1}$. Interestingly, with $-\widehat{S}_{22}^{-1}$
(see \cref{fig:cfm}), the staircasing effect is clear, where every
alternate iteration makes little to no reduction in residual. Although
convergence has some sawtooth character for block-diagonal with
$\widehat{S}_{22}^{-1}$ as well, it is much weaker, and the
staircasing effect is not truly observed. It is possible this explains
the slightly better convergence obtained with $\widehat{S}_{22}^{-1}$
in \cref{fig:5m}, regardless of momentum relative-residual tolerance.

\begin{figure}[!ht]
    \centering
    \begin{center}
        \begin{subfigure}[t]{0.475\textwidth}
            \includegraphics[width=\textwidth]{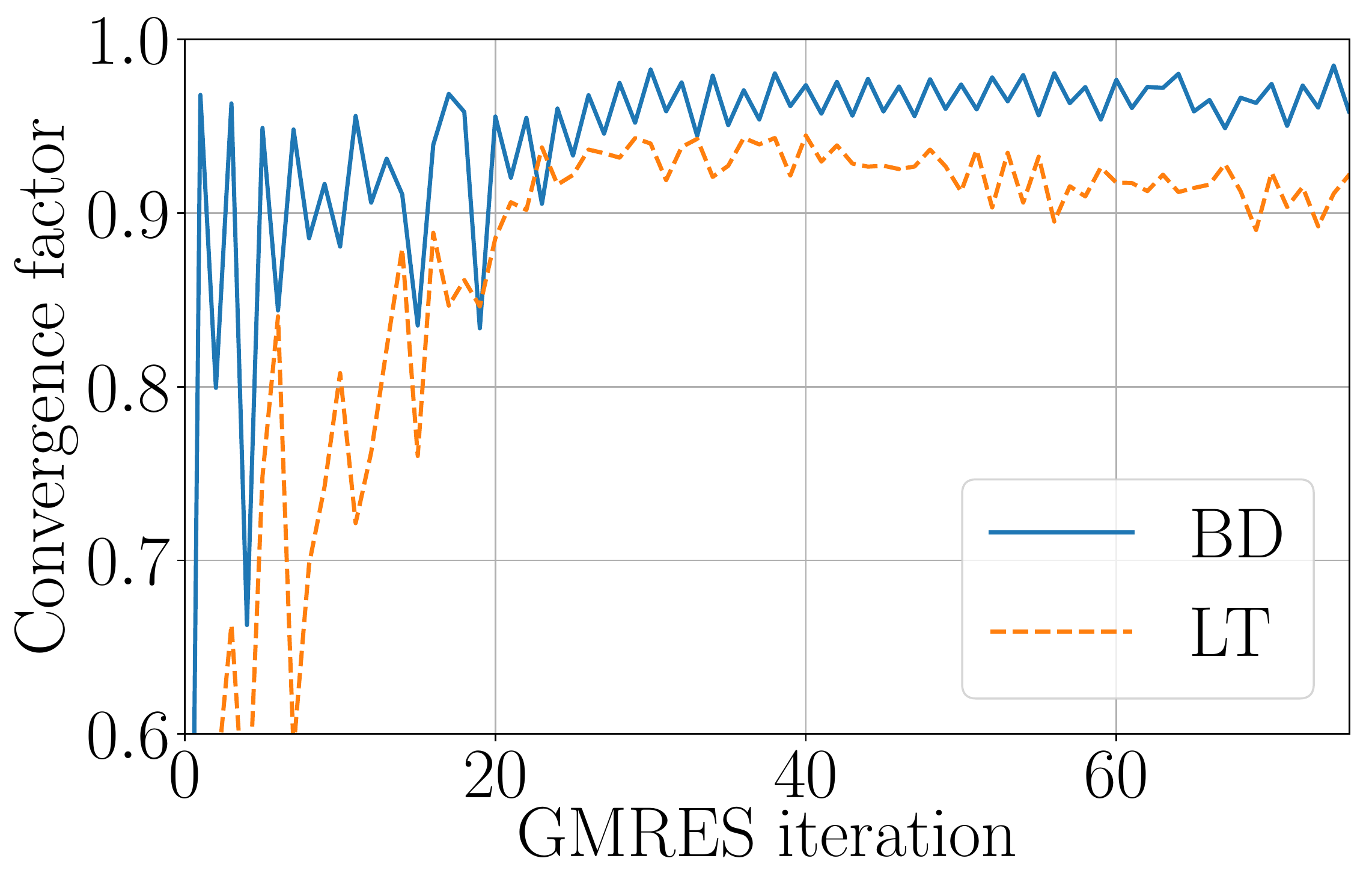}
	\caption{Preconditioning using $\widehat{S}_{22}^{-1}$.}
	\label{fig:cfp}
        \end{subfigure}
        \begin{subfigure}[t]{0.475\textwidth}
            \includegraphics[width=\textwidth]{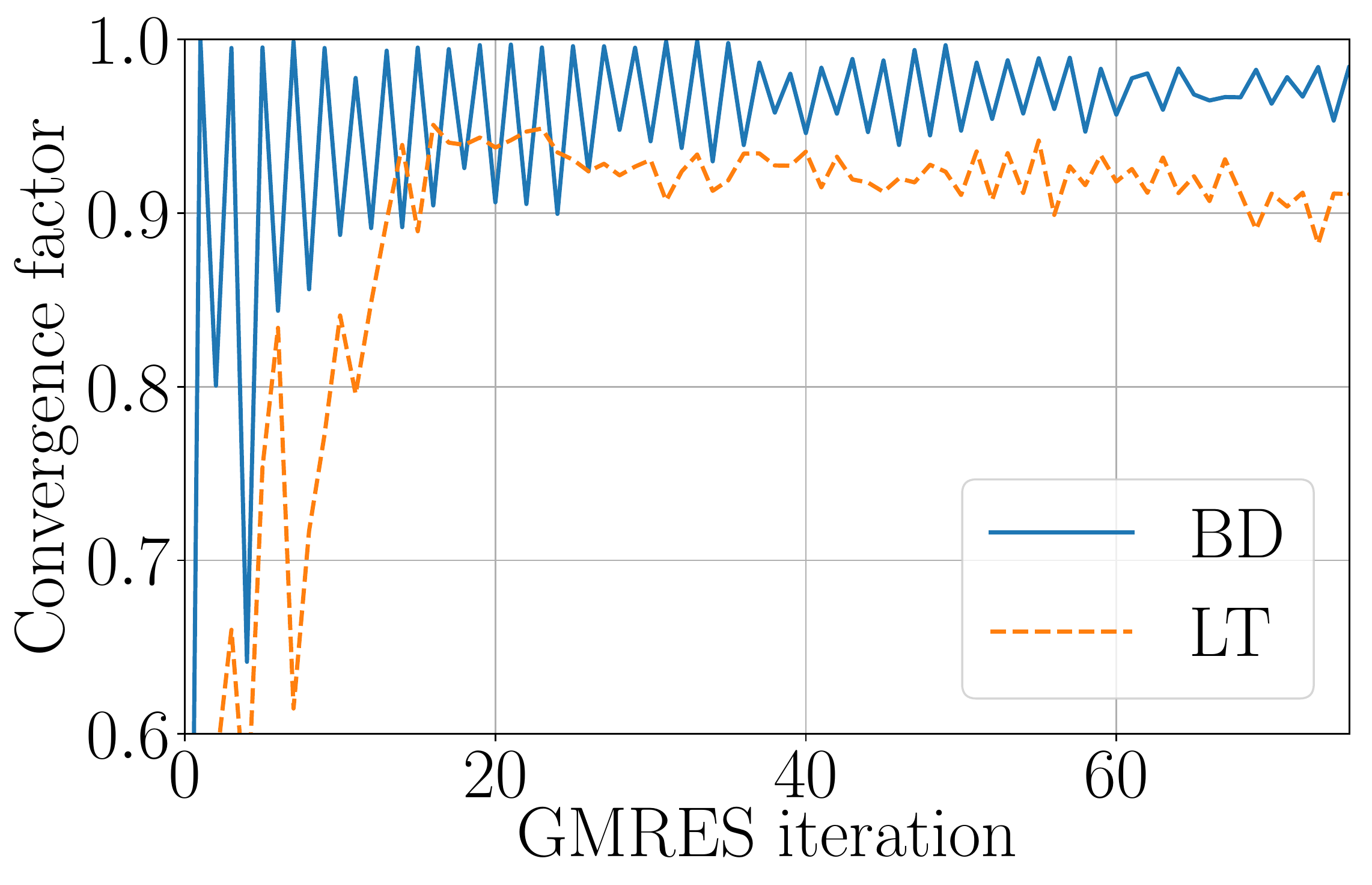}
	\caption{Preconditioning using $-\widehat{S}_{22}^{-1}$.}
	\label{fig:cfm}
        \end{subfigure}
    \end{center}
    \caption{Convergence factor as a function of {\color{black} FGMRES} iteration for block-diagonal (BD) and
    block lower-triangular (LT) preconditioning. \Cref{fig:cfp} uses the natural sign on $\widehat{S}_{22}^{-1}$,
    while \cref{fig:cfm} adds a negative to $\widehat{S}_{22}^{-1}$.}
    \label{fig:CF}
\end{figure}

\section{Conclusions}
\label{sec:conclusions}

This paper analyzes the relationship between Krylov methods with
$2\times 2$ block preconditioners and the underlying preconditioned
Schur complement. Under the assumption that one of the diagonal blocks
is inverted exactly, we prove a direct relationship between the
minimizing Krylov polynomial of a given degree for the two systems,
thereby proving their equivalence and the fact that an effective Schur
complement preconditioner is a necessary and sufficient condition for
an effective $2\times 2$ block preconditioner. Theoretical results
give further insight into choice of block preconditioner, including
that (i) symmetric block-triangular and approximate block-LDU
preconditioners offer a minimal reduction in iteration count over
block-triangular preconditioners, at the expense of additional
computational cost, and (ii) block-diagonal preconditioners take about
twice as many iterations to reach a given residual tolerance as
block-triangular preconditioners.

Numerical results on an HDG discretization of the steady linearized
Navier--Stokes equations confirm the theoretical contributions, and
show that the practical implications extend to the case of a
Schur-complement approximation coupled with an inexact inverse of the
other diagonal block.
{\color{black}
Although not shown here, it is worth pointing out we have observed similar
results with inexact block preconditioners in other applications. For HDG
discretizations of symmetric Stokes and Darcy problems, if the pressure Schur
complement is approximated by a spectrally equivalent operator, applying two
to four multigrid cycles to the momentum block yields comparable convergence
on the larger $2\times 2$ system as applying a direct solve on the momentum block.
Classical source iteration and DSA preconditioning for S$_N$ discretizations of
neutron transport can also be posed as a $2\times 2$ block preconditioning
\cite{19hetdsa}. There we have also observed that when applying AMG iterations
to the (1,1)-block and Schur complement approximation, only 2-3 digits of
residual reduction yields convergence on the larger $2\times 2$ system in a
similar number of iterations as applying direct solves to each block. In
each of these cases, convergence of minimal residual methods applied to the
$2\times 2$ system is defined by the preconditioning of the Schur complement. 
}

\appendix
\section{Proof of \Cref{th:ldu_cg}}
\label{ap:proof_ldu_cg}

\begin{proof}[Proof of \Cref{th:ldu_cg}]
  
  Here we prove the case of $k=2$. An analogous derivation appealing
  to \cref{eq:LDU11} yields equivalent results for $k=1$.

  Recall that CG forms a minimizing consistent polynomial of
  $M_{22}^{-1}A$ in the $A$-norm.  Let $\varphi^{(d)}$ be the
  minimizing polynomial of degree $d$ in $M_{22}^{-1}A$ for error
  vector $\mathbf{e}$ in the $A$-norm. Then, expressing $A$ in a block
  LDU sense to simplify the term $A\varphi^{(d+1)}(M_{22}^{-1}A)$,
  we immediately obtain a lower bound:
    \begin{align*} 
    \|\varphi^{(d)}(M_{22}^{-1}A)\mathbf{e}\|_A^2 & = \langle A\varphi^{(d)}(M_{22}^{-1}A)\mathbf{e}, \varphi^{(d)}(M_{22}^{-1}A)\mathbf{e}\rangle \\
                                                  &\hspace{-12ex} = \left\langle \begin{bmatrix}A_{11} & \mathbf{0} \\ \mathbf{0} & S_{22}\end{bmatrix}
                                                                                                                                    \begin{bmatrix} \varphi^{(d)}(I)(\mathbf{e}_1 + A_{11}^{-1}A_{12}\mathbf{e}_2) \\ \varphi^{(d)}(\widehat{S}_{22}^{-1}S_{22})\mathbf{e}_2\end{bmatrix},
    \begin{bmatrix} \varphi^{(d)}(I)(\mathbf{e}_1 + A_{11}^{-1}A_{12}\mathbf{e}_2) \\ \varphi^{(d)}(\widehat{S}_{22}^{-1}S_{22})\mathbf{e}_2t\end{bmatrix}
    \begin{bmatrix} \mathbf{e}_1\\\mathbf{e}_2\end{bmatrix}\right\rangle \\
                                                  &\hspace{-12ex} \geq \langle S_{22} \varphi^{(d)}(\widehat{S}_{22}^{-1}S_{22})\mathbf{e}_2, \varphi^{(d)}(\widehat{S}_{22}^{-1}S_{22})\mathbf{e}_2 \rangle \\
                                                  &\hspace{-12ex} = \|\varphi^{(d)}(\widehat{S}_{22}^{-1}S_{22})\mathbf{e}_2\|_{S_{22}}^2 \\
                                                  &\hspace{-12ex} = \|\varphi_{22}^{(d)}(\widehat{S}_{22}^{-1}S_{22})\mathbf{e}_2\|_{S_{22}}^2,
  \end{align*}
    where $\varphi_{22}^{(d)}$ is the minimizing polynomial of degree
  $d$ in $\widehat{S}_{22}^{-1}S_{22}$ for error vector
  $\mathbf{e}_2$.  A lower bound on the minimizing polynomial of
  degree $d$ in norm follows immediately by noting that
    \begin{align*}
    \|\varphi^{(d)}(M_{22}^{-1}A)\|_A &= \sup_{\mathbf{e} \neq\mathbf{0}} \frac{\|\varphi^{(d)}(M_{22}^{-1}A)\mathbf{e}\|_A}{\|\mathbf{e}\|_A} 
                                        \geq \sup_{\substack{\mathbf{e}_2 \neq\mathbf{0},\\\mathbf{e}_1=-A_{11}^{-1}A_{12}}} \frac{\|\varphi^{(d)}(M_{22}^{-1}A)\mathbf{e}\|_A}{\|\mathbf{e}\|_A} \\
                                      &\hspace{3ex} = \|\varphi^{(d)}(\widehat{S}_{22}^{-1}S_{22})\|_{S_{22}} \geq \|\varphi_{22}^{(d)}(\widehat{S}_{22}^{-1}S_{22})\|_{S_{22}}.
  \end{align*}
  
  For an upper bound, let $\varphi_{22}^{(d)}$ be the minimizing
  polynomial of degree $d$ in $\widehat{S}_{22}^{-1}S_{22}$ for error
  vector $\mathbf{e}_2$ in the $S_{22}$-norm. Define the degree $d+1$
  polynomial $q(t):= (1-t)\varphi^{(d)}_{22}(t)$, and let
  $\varphi^{(d+1)}$ be the minimizing polynomial of degree $d+1$ in
  $M_{22}^{-1}A$ for error vector $\mathbf{e}$ in the $A$-norm.  Then
\begin{align*}
\|\varphi^{(d+1)}(M_{22}^{-1}A)\mathbf{e}\|_A^2 & \leq \|q(M_{22}^{-1}A)\mathbf{e}\|_A^2 \\
& = \left\langle \begin{bmatrix}A_{11} & \mathbf{0} \\ \mathbf{0} & S_{22}\end{bmatrix}
	\begin{bmatrix} \mathbf{0} & \mathbf{0} \\ \mathbf{0} & q(\widehat{S}_{22}^{-1}S_{22})\end{bmatrix}
	\begin{bmatrix} \mathbf{e}_1\\\mathbf{e}_2\end{bmatrix},
	\begin{bmatrix} \mathbf{0} & \mathbf{0} \\ \mathbf{0} & q(\widehat{S}_{22}^{-1}S_{22})\end{bmatrix}
	\begin{bmatrix} \mathbf{e}_1\\\mathbf{e}_2\end{bmatrix}\right\rangle \\
& = \langle S_{22} q(\widehat{S}_{22}^{-1}S_{22})\mathbf{e}_2, q(\widehat{S}_{22}^{-1}S_{22})\mathbf{e}_2 \rangle \\
& = \|(I - \widehat{S}_{22}^{-1}S_{22})\varphi_{22}^{(d)}(\widehat{S}_{22}^{-1}S_{22})\mathbf{e}_2\|_{S_{22}}^2.
\end{align*}
For a bound in norm, note that for a fixed $\mathbf{e}_2$, $\|\mathbf{e}\|_A^2$ is a quadratic function in $\mathbf{e}_1$,
with minimum obtained at $\mathbf{e}_1 := -A_{11}^{-1}A_{12}$. Then,
\begin{align*}
\|\varphi^{(d+1)}(M_{22}^{-1}A)\|_A^2 & = \sup_{\mathbf{e}\neq\mathbf{0}} \frac{\|\varphi^{(d+1)}(M_{22}^{-1}A)\mathbf{e}\|_A^2}{\|\mathbf{e}\|_A^2}
	\leq \sup_{\mathbf{e}\neq\mathbf{0}} \frac{ \|q(M_{22}^{-1}A)\mathbf{e}\|_A^2}{\|\mathbf{e}\|_A^2} \\
&\hspace{-17ex} = \sup_{\mathbf{e}_2} \frac{\|(I - \widehat{S}_{22}^{-1}S_{22})\varphi_{22}^{(d)}(\widehat{S}_{22}^{-1}S_{22})\mathbf{e}_2\|_{S_{22}}^2}
	{\inf_{\mathbf{e}_1} \|\mathbf{e}\|_A^2} 
= \sup_{\mathbf{e}_2} \frac{\|(I - \widehat{S}_{22}^{-1}S_{22})\varphi_{22}^{(d)}(\widehat{S}_{22}^{-1}S_{22})\mathbf{e}_2\|_{S_{22}}^2}
	{\|\mathbf{e}_2\|_{S_{22}}^2} \\
& \hspace{0ex}= \|(I - \widehat{S}_{22}^{-1}S_{22})\varphi_{22}^{(d)}(\widehat{S}_{22}^{-1}S_{22})\|_{S_{22}}^2.
\end{align*}

\end{proof}

\bibliographystyle{siamplain}
\bibliography{references}
\end{document}